\documentclass[final,12pt]{colt2023} 
\usepackage{amsmath, amssymb, mathrsfs}

\newcommand{\reals}{\mathbb{R}}
\newcommand{\prettyref}{\eqref}
\newcommand{\calX}{\mathcal{X}}

\newcommand{\TV}{\mathrm{TV}}
\newcommand{\KL}{\mathrm{KL}}

\newcommand{\PP}{\mathbb{P}}

\newcommand{\QQ}{\mathbb{Q}}
\newcommand{\RR}{\mathbb{R}}

\newcommand{\EE}{\mathbb{E}}
\newcommand{\SSS}{\mathbb{S}}
\newcommand{\mmu}{\mathbf{u}}
\newcommand{\mx}{\mathbf{x}}
\newcommand{\my}{\mathbf{y}}
\newcommand{\calN}{\mathcal{N}}
\newcommand{\supp}{\mathtt{supp}}

\def\calP{\mathcal{P}}
\def\calX{\mathcal{X}}

\title[Entropic characterization of optimal rates for learning Gaussian mixtures]{Entropic characterization of optimal rates for learning Gaussian mixtures}
\usepackage{times}

\coltauthor{%
 \Name{Zeyu Jia} \Email{zyjia@mit.edu}\\
 \addr Massachusetts Institute of Technology
 \AND
 \Name{Yury Polyanskiy} \Email{yp@mit.edu}\\
 \addr Massachusetts Institute of Technology
 \AND
 \Name{Yihong Wu} \Email{yihong.wu@yale.edu}\\
 \addr Yale University
}

\begin{document}

\maketitle

\begin{abstract}%

  We consider the question of estimating multi-dimensional Gaussian mixtures (GM) with compactly
  supported or subgaussian mixing distributions. Minimax estimation rate for this class (under
  Hellinger, TV and KL divergences) is a long-standing open question, even in one dimension. In this paper we
  characterize this rate (for all constant dimensions) in terms of the metric entropy of the class. Such
  characterizations originate from seminal works of \cite{lecam1973convergence,
  birge1983approximation,haussler1997mutual,yang1999information}.
  However, for GMs a key ingredient missing from earlier work (and widely sought-after) is a comparison result showing that the KL and the squared Hellinger distance are within a constant multiple of each other uniformly over the
  class. Our main technical contribution is in showing this fact, from which
  we derive entropy characterization for estimation rate under Hellinger and
  KL. Interestingly, the sequential (online learning) estimation rate is characterized by the
  global entropy, while the single-step (batch) rate corresponds to local entropy, paralleling a
  similar result for the Gaussian sequence model recently discovered by~\cite{neykov2022minimax} and \cite{jaonad2023coding}. Additionally, since Hellinger is a proper metric, our comparison shows that GMs under KL satisfy the triangle inequality within multiplicative constants, implying that proper and improper estimation
  rates coincide.
\end{abstract}


\section{Introduction}
\par Gaussian mixtures are among the most popular and useful classes of distributions for modeling real data with heterogeneity.
Specifically, each $d$-dimensional \textit{mixing distribution} $\pi$ induces a  Gaussian \textit{mixture} $f_\pi$, which is  the convolution of $\pi$ with the $d$-dimensional standard Gaussian distribution $\mathcal{N}(0, I_d)$, namely
$$f_\pi({x}) = (\pi * \varphi)({x}) = \int_{\reals^d} 
 \varphi({x} - {z}) \pi(d{z}),$$
where $\varphi({z}) = \frac{1}{\sqrt{2\pi}^d}\exp\left(-\|{z}\|_2^2/2\right)$ is the standard normal density. 

There is a vast literature in statistics and machine learning on various aspects of mixture models such as parameter estimation and clustering. In this paper we focus on learning the mixture model in the sense of density estimation. To this end, it is necessary to impose tail conditions on the mixing distribution. Specifically, we consider two classes of Gaussian mixtures classes, wherein the mixing distribution is either compactly supported or subgaussian.



 To measure the density estimation error, it is common to use $f$-divergences, notably, Kullback-Leibler (KL) divergence
 $\KL(f\|g) = \int f \log\frac{f}{g}$, the squared Hellinger distance 
 $H^2(f, g) = \int (\sqrt{f}-\sqrt{g})^2$, and the total variation distance $\TV(f,g) = \frac{1}{2} \int |f-g|$. In this paper, we are chiefly concerned with mainly focus on the estimation rates under the KL-divergence and the squared Hellinger distance, 
as opposed to the $L_2$ distance due to its lack of operational meaning.\footnote{Indeed, for densities supported on the entire real line, it is possible that two densities are arbitrarily close in $L_2$ distance but separated by a large  $\TV$ distance and hence easily distinguishable. In fact, for the entire class of Gaussian mixtures, \cite{kim2014minimax} showed ignoring the mixture structure and simply applying the kernel density estimator designed for analytic densities \cite{ibragimov2001estimation} achieves the optimal rate in $L_2$. On the other hand, consistent estimation of Gaussian mixtures in more meaningful loss function such as $\TV$ is impossible unless tail conditions on the mixing distribution are imposed.}

Estimating Gaussian mixture densities is a classical topic in nonparametric statistics. 
Under the  squared Hellinger loss, the minimax lower bound $\Omega((\log n)^d /n)$ was proved in~\cite{kim2014minimax,kim2022minimax} the lower bound for subgaussian mixing distributions. On the constructive side, nonparametric maximum likelihood estimator (NPMLE) and sieve MLE have been analyzed in 
\cite{van1993hellinger, wong1995probability, van1996rates,genovese2000rates, ghosal2001entropies, ghosal2007posterior, zhang2009generalized}. In particular, the NPMLE, which offers a practical algorithm for properly learning the mixture model, is shown to achieve a near-parametric rate of 
$O((\log n)^2/n)$ for the subgaussian case \cite{zhang2009generalized}, which is subsequently generalized to $O(\log^{d+1} n/n)$ in $d$ dimensions
\cite{saha2020nonparametric}. 
Similar results for compactly
supported mixing distribution are also obtained in \cite[Theorem 20]{polyanskiy2021sharp}
Despite these advances, determining the optimal rate remains a long-standing open question even in one dimension.


Departing from maximum likelihood, there is a long line of work that aims at characterizing density estimation rates in terms of metric entropy of the class. These general entropic upper bounds originate  from the seminal work of \cite{lecam1973convergence, birge1983approximation, birge1986estimating} for the Hellinger loss,  \cite{yatracos1985rates} for the $\TV$ loss, and \cite{yang1999information} for the KL loss. 
(We refer to in~\cite[Chapter 33]{polyanskiy2022information}
for a detailed exposition on these results.)
On the other hand, entropic lower
bounds for KL and Hellinger losses were established for both the batch~\cite{haussler1997mutual} and sequential estimation~\cite{yang1999information}.  
However, these entropy-based upper and lower bounds in general do not match unless extra conditions are imposed on the behavior on the model class (those conditions are satisfied, most notably, for the H\"older density class on $[0,1]^d$). Notably, a simple condition that ensures a sharp entropic determination of the minimax rate is the comparability of the Hellinger and KL divergence, namely, for any density $f$ and $g$ in the model class:
\begin{equation}
    \label{eq:HKL}
    \KL(f\|g) \asymp H^2(f, g) 
\end{equation}
where $\asymp$ denotes equality within constant multiplicative factors.
Note that the one-sided inequality $\KL(f\|g) \geq H^2(f, g)$ is always true cf.~e.g.~\cite[Eq.~(7.30)]{polyanskiy2022information}.
As such, whenever $\KL$ is dominated by $H^2$, the sharp minimax rate is determined by the local Hellinger entropy of the model class.

Indeed, this entropy-based approach has been successfully taken in~\cite{doss2020optimal} to determine the sharp rate for the special case of \textit{finite-component} Gaussian mixtures in general dimensions. Specifically, for the class of $k$-component GMs, \cite[Theorem 4.2]{doss2020optimal} shows that
\begin{equation}
    \label{eq:HKL-kGM}
    \KL(f_\pi\|f_\eta) \asymp_k H^2(f_\pi,f_\eta),
\end{equation}
where $\pi$ and $\eta$ are $k$-atomic distributions supported on a Euclidean ball of bounded radius in $\reals^d$ and $\asymp_k$ hides constants depending only on $k$.
The proof of this result is based on the method of moments which shows both distances are proportional to the Euclidean distance between the moment tensors of mixing distributions up to degree $2k-1$. 
The crucial part of \eqref{eq:HKL-kGM} is that it does not depend on the ambient dimension $d$. As such, this allows the optimal squared Hellinger rate to be determined by the local entropy, which, in turn, can be tightly estimated via the low rank of the  moment tensor, leading to the sharp rate of $\Theta_k(\frac{d}{n})$ that holds even in high dimensions.
On the other hand, \eqref{eq:HKL-kGM} is not fully dimension-free in that the proportionality constant therein is in fact \textit{exponential} in $k$, the number of components, a limitation of the moment-based approach.
As such, it is unclear whether \eqref{eq:HKL-kGM} continues to hold for continuous GMs even in one dimension. 

\par We review related results on upper-bounding the KL divergence by Hellinger distance. \cite[Lemma 5]{birge1998minimum} shows that $D_{KL}(f\|g)\lesssim H^2(f, g)$ if $\mathrm{ess}\sup\frac{df}{dg} < \infty$. This results was further generalized to $\alpha$-generalized Hellinger divergence in \cite[Theorem 9]{sason2016f}. However, ratios between two Gaussian mixture densities are not bounded. \cite[Theorem 5]{wong1995probability} points out that if $\int_{f/g\ge \exp(1/\delta)}f^{\delta+1}/g^\delta < \infty$ for some $\delta > 0$, then we have $D_{KL}\lesssim H^2 \log (1/H^2)$. 
This method was extended by~\cite{haussler1997mutual} and we also use it in our Theorem~\ref{thm-KL-nd}. Note, however, that this method is unable to produce a linear upper bound: $D_{KL}(f\|g)\lesssim H^2(f, g)$.  Yet another result follows by choosing $\eta = 1/2$ and $\bar{\eta} = 1$ in \cite[Lemma 13]{grunwald2020fast}, which proves $D_{KL}\le c_u H^2(f, g)$ with $c_u = \frac{u+2}{c}$ provided that $f, g\in \mathcal{F}$ and $\mathcal{F}$ satisfies the so-called $(u, c)$-witness condition, i.e. 
$\int f\log (f/g)\mathbf{1}_{f/g\le \exp(u)}\ge c\cdot \int f\log (f/g)$. However, Gaussian mixtures again do not satisfy this condition. In particular, the left side of the above inequality can even be negative in some cases.\footnote{To see this, consider $f = \mathcal{N}(0, 1)$ and $g = \mathcal{N}(-\delta, 1)$ with $\delta > 0$. Then $\int f\log (f/g)\mathbf{1}_{f/g\le \exp(u)} = \frac{\delta^2}{2} - \frac{1}{\sqrt{2\pi}}\delta\exp(-u^2/2)$. Hence for any choice of $u$, there always exists a $\delta$ close to zero such that this integral is negative.}

In this paper we resolve the  question of KL to Hellinger comparison and show that with a constant factor that depends (at most linearly) on the dimension, by proving that
\begin{equation}
    \label{eq:HKL2}
    \KL(f_\pi\|f_\eta) \asymp_d H^2(f_\pi,f_\eta),
\end{equation}
where $\pi$ and $\eta$ are arbitrary distributions supported on a bounded ball in $\reals^d$; furthermore, this result can be made dimension-free with an extra logarithmic factor. In addition, we show that \eqref{eq:HKL2} holds for $(1-\epsilon)$-subgaussian mixing distributions but fails for $(1+\epsilon)$-subgaussian  distributions.  
Curiously, our method does not rely on comparing moments of mixing measures, the prevailing method for analyzing
statistical distances between mixture distributions (cf.~e.g.~\cite{wu2020optimal,wu2020polynomial,
bandeira2020optimal,
fan2020likelihood,doss2020optimal,chen2021asymptotics}).

The new comparison result has various statistical consequences, of which we report here one (see Corollary~\ref{cor-1}). 
To estimate the GM density with compactly supported or $(1-\epsilon)$-subgaussian mixing distributions based on an iid sample of size $n$, the minimax proper or improper density estimation risks under KL divergence or squared Hellinger distance are tightly characterized by the \textit{local} Hellinger entropy of the density class, thereby reducing the question of optimal rates to that of computing the local entropy. 
Furthermore, the minimax risks in the sequential version (as opposed to the batch setting above) of this problem are tightly characterized by the \textit{global} Hellinger entropy of the class. 
A similar phenomenon of local-vs-global entropy
 has been observed in a pair of recent works on Gaussian sequence
model:~\cite{neykov2022minimax} showed that batch risk is controlled by the local entropy and
Mourtada~\cite{jaonad2023coding} showed that sequential risk is controlled by the global entropy.

\paragraph{Notation} Let $B_2(r) = \left\{\mx\in\RR^d:\|\mx\|_2\le r\right\}$ denote the Euclidean ball of radius $r$ centered at 0. Denote by $\supp(\pi)$ the support of a probability measure $\pi$.
We say a distribution $\pi$ on $\reals^d$ is $K$-subgaussian if for $X\sim \pi$,
$$\mathbf{P}[\|X\|_2 > t]\le \exp\left(-\frac{t^2}{2K^2}\right), \quad \forall t\ge 0.$$

\paragraph{Organization}
The rest of the paper is organized as follows. Section \ref{sec:main} states our main results by providing upper bounds on KL divergence according to squared Hellinger for Gaussian mixtures. To illustrate the main ideas, the proof for one dimension is provide in Section \ref{sec-1d} as a warm-up. The dimension-free bound for Gaussian mixtures where mixing distribution is compactly supported is provided in Section \ref{sec-2}. Finally, the proof of Corollary \ref{cor-1}, showing that the estimation rates are tightly characterized by Hellinger entropies, is given in Section \ref{sec-proof-cor1}. Proofs of other results are deferred to appendices.


\section{Main Results}\label{sec:main}


Before discussing their statistical consequences, we first state the main comparison results that control the KL divergence between Gaussian mixtures using their Hellinger distance.


For compactly supported mixing distributions, our main result is as follows:
\begin{theorem}\label{thm-KL-d}
Let $\pi$ and $\eta$ be supported on $B_2(M)$ in $\reals^d$ where $M\ge 2$. Then
$$\KL(f_\pi\|f_\eta)\le 5154 (M^2\vee d) H^2(f_\pi,f_\eta).$$
\end{theorem}
In Section \ref{sec-1d} we provide a proof of Theorem \ref{thm-KL-d} in one dimension. The proof of general cases is included in Appendix \ref{app-1}. 
\begin{remark}
The bound in Theorem \ref{thm-KL-d} is tight up to constant factors depending on the dimension $d$. To see this, consider $\pi = \delta_\mmu$ where $\mmu = (M, 0, 0, \cdots, 0)$ and $\eta = \delta_{-\mmu}$. Then we have $f_\pi = \mathcal{N}(\mmu, I)$ and $f_\eta = \mathcal{N}(-\mmu, I)$. A direct computation shows that 
$\KL(f_\pi\|f_\eta) = \frac{\|\mmu - (-\mmu)\|^2}{2} = 2M^2$, 
	while 
	$H^2(f_\pi, f_\eta) = 2 - 2\exp\left(-\frac{M^2}{2}\right)\le 2$.

\end{remark}

Complementing  Theorem \ref{thm-KL-d}, we also have the following dimension-free upper bound at the price of a mere logarithmic factor. This theorem is a direct corollary of \cite[Theroem 5]{wong1995probability}, if we notice that $f_\pi, f_\eta$ satisfy their condition 
$\int_{f_\pi/f_\eta\ge e}f_\pi(f_\pi/f_\eta)\le \exp(4M^2) < \infty$ for any $\pi, \eta$ supported on $B_2(M)$. For completeness, we include a proof in Section \ref{sec-2}:
\begin{theorem}\label{thm-KL-nd}
Let $\pi$ and $\eta$ be supported on $B_2(M)$ in $\reals^d$ where $M\ge 1$. Then
$$\KL(f_\pi\|f_\eta)\le 200M^2H^2(f_\pi, f_\eta) +  16H^2(f_\pi, f_\eta)\log\frac{1}{H^2(f_\pi, f_\eta)}.$$
\end{theorem}

Next we consider the class of subgaussian mixing distributions.
We discover a dichotomy depending on the subgaussian constant $K$: When $K < 1$, the KL divergence is indeed proportional to the squared Hellinger distance. When $K > 1$, such upper bound does not exist.
\begin{theorem}\label{thm-sub-ub}
	Let $\pi, \eta$ be two $d$-dimensional $K$-subgaussian distributions where $K < 1$. Then
	$$\KL(f_\pi\|f_\eta)\le 1660056\left(\frac{1}{(1-K)^3}\vee 8d^3\right)H^2(f_\pi, f_\eta).$$
\end{theorem}
\begin{theorem}\label{thm-sub-lb}
	Fix $K > 1$. For any $C > 0$, there exists a $1$-dimensional $K$-subgaussian distribution $\pi$ such that
	$$\KL(f_\pi\|\mathcal{N}(0, 1))\ge C\cdot H^2(f_\pi, \mathcal{N}(0, 1)).$$
\end{theorem}

\begin{remark}
	Notice that this phenomenon of dichotomy of $K > 1$ and $K < 1$ for Gaussian mixtures with  $K$-subgaussian mixing distribution was also observed in~\cite{block2022rate}. Therein, it is shown that the convergence rate of smoothed $n$-point empirical distribution to the smoothed population distribution under
Wasserstein distance is $O(1/\sqrt{n})$ for $K < 1$, and $\omega(1/\sqrt{n})$ for $K > 1$.  
\end{remark}

We further have the following dimension-free upper bound that holds for all $K > 0$.
\begin{theorem}\label{thm-sub-ub-nd}
Let $\pi$ and $\eta$ be $K$-subgaussian distributions on $\reals^d$. Then
$$\KL(f_\pi\|f_\eta)\le (10240K^4+652)H^2(f_\pi, f_\eta)\log\frac{4}{H^2(f_\pi, f_\eta)}.$$
\end{theorem}

The results presented so far are structural results on the information geometry of Gaussian mixture, whose proof are included in Appendix \ref{app-1}-\ref{app-5}. Next we discuss their statistical consequences. 
We start with the definition of covering/local covering number and minimax risks of density estimation.

\begin{definition}[Covering Number and Local Covering Number]
	Let $\mathcal{P}$ be a set of distributions over some measurable space $\mathcal{X}$. The Hellinger covering number of $\calP$ is
	$$\mathcal{N}_{H}(\mathcal{P}, \epsilon)\triangleq\min\left\{N: \exists Q_1, \cdots, Q_n\in \Delta(\mathcal{X}),\ \ \sup_{P\in\mathcal{P}}\inf_{1\le i\le N}H(P, Q_i)\le \epsilon\right\},$$
	where $\Delta(\mathcal{X})$ denotes the collection of all probability distributions on $\calX$.
	The local Hellinger covering number of $\calP$ is
	$$\mathcal{N}_{loc, H}(\mathcal{P}, \epsilon) \triangleq \sup_{P\in\mathcal{P}, \eta\ge \epsilon}\mathcal{N}_H(B_H(P, \eta)\cap \mathcal{P}, \eta/2),$$
	where $B_H(P, \eta)$ is the Hellinger ball of radius $\eta$ centered at $P$.
\end{definition}
\par We further define the minimax risks for proper and improper density estimation as well as the minimax risk in a sequential setting.

\begin{definition}[Proper and Improper Density Estimation Minimax Risk]
\label{def-batch}
	For a given class $\mathcal{P}$ of distributions over $\mathcal{X}$, we define the improper minimax risk
 $R_{H^2, n}, R_{KL, n}$ and the proper minimax risk
	$\tilde{R}_{KL, n}$ with sample size $n$ as follows: for $d\in \{H^2, \KL\}$, we define
	$$R_{d, n}(\mathcal{P}) \triangleq \inf_{\hat{f}_n}\sup_{f\in\mathcal{P}}\mathbb{E}_f\left[d(f, \hat{f}_n)\right],$$
	and also\footnote{Since Hellinger distance is a valid metric, for proper and improper  density estimation, the minimax squared Hellinger risks coincide within a factor of four, as any estimator can be made proper by its Hellinger projection on the model class.}
	$$\tilde{R}_{KL, n}(\mathcal{P}) \triangleq \inf_{\hat{f}_n\in\mathcal{P}}\sup_{f\in\mathcal{P}}\mathbb{E}_f\left[\KL(f\|\hat{f}_n)\right],$$
	where $\hat{f}_n(\cdot) = \hat{f}_n(\cdot; X_1,\ldots,X_n)$ is a density estimator based on $X_1,\ldots,X_n$ drawn iid from $P$.
\end{definition}

\begin{definition}[Sequential Density Estimation Minimax Risk (Improper)]\label{def-sequential}
	For a given class $\mathcal{P}$ of distributions over $\mathcal{X}$, the sequential minimax risks $C_{H^2, n}$ and $C_{KL, n}$ are defined as: for $d\in \{H^2, \KL\}$,
	$$C_{d, n}(\mathcal{P}) = \inf_{\hat{f}_1, \cdots, \hat{f}_n}\sup_{f\in\mathcal{P}}\sum_{t=1}^N \mathbb{E}[d(f, \hat{f}_t(\cdot|X_1, \cdots, X_{t-1}))]$$
	where $\hat{f}_t: \mathcal{X}^{t-1}\to \Delta(\mathcal{X})$ denotes the density estimator at time $t$ based on observations $X_1, \cdots, X_{t-1}$.
\end{definition}

We refer to Definitions \ref{def-batch} and \ref{def-sequential} as the batch and online settings, respectively. 
The following corollary shows that the minimax density estimation risks in these settings can be characterized by the local and global Hellinger entropy up to constant factors.
Furthermore, we show that proper and improper density estimation rates coincide. As explained earlier this is well-known for Hellinger loss but far from clear for KL loss which does not satisfy triangle inequality. In fact, the celebrated Yang-Baron construction \cite{yang1999information} produces an improper density estimate. Nevertheless, we show that for Gaussian mixture class there is no gain in stepping outside the model class.

\begin{corollary}\label{cor-1}
Let $\mathcal{P}_{com}(M)$ and $\mathcal{P}_{sub}(K)$ denote the collection of $d$-dimensional Gaussian mixtures where the mixing distribution is supported on $B_2(M)$ and $K$-subgaussian, respectively, i.e.,
\begin{align*}
	\mathcal{P}_{com}(M) & = \{\pi * \mathcal{N}(0, I_d)|\mathrm{supp}(\pi)\subset B_2(M)\},\\
	\mathcal{P}_{sub}(K) & = \{\pi * \mathcal{N}(0, I_d)|\pi \text{ is } K\text{-subgaussian}\}.
\end{align*}
Then for any compact (under Hellinger) subset $\mathcal{P}$ where $\mathcal{P}\subset \mathcal{P}_{com}(M)$ or $\mathcal{P}\subset \mathcal{P}_{sub}(K)$ with $K < 1$, we have 
the following characterization on the proper or improper minimax risk:
$$ R_{H^2, n}(\mathcal{P}) \asymp R_{KL,n}(\mathcal{P}) \asymp \tilde R_{KL,n}(\mathcal{P}) \asymp \inf_{\epsilon>0} \epsilon^2 + {1\over n} \log \calN_{loc,H}(\calP, \epsilon),$$
and also for sequential minimax risk:
$$ C_{H^2, n}(\mathcal{P}) \asymp C_{KL,n}(\mathcal{P}) \asymp \inf_{\epsilon>0} n\epsilon^2 + \log \calN_{H}(\calP, \epsilon).$$
Here $\asymp$ hides constants that may depend on $M, K$, or $d$ but not on $n$.
\end{corollary}

As we mentioned previously, a recent pair of works~\cite{neykov2022minimax,jaonad2023coding}
established the same phenomenon: the sequential rate is given by global entropy, while the batch
rate is given by the local entropy, though,  their work is for a very different setting of a
Gaussian sequence model.

\par Apart from the KL divergence and Hellinger distance, we also obtained comparison results for other distances between distributions, e.g. $\chi^2$-divergence, TV and $L_2$ distances. See Appendix \ref{sec:other_dist}.





We close this section with a list of related open problems.
\begin{enumerate}

	\item \textbf{Fully dimension-free comparison:} Currently our upper bound on KL/$\chi_2$ according to Hellinger is depending on the dimension of the distributions, can we remove this dependence on the dimension? Suppose $\pi$ and $\eta$ are two $d$-dimensional distributions supported on $B_2(M)$, is there a constant $C_{KL}(M), C_{\chi_2}(M)$ such that
	\begin{align*}
		\KL(f_\pi\|f_\eta) & \le C_{KL}(M)\cdot H^2(f_\pi, f_\eta),
	\end{align*}
 which would have the best of both worlds of \prettyref{eq:HKL-kGM} and \prettyref{eq:HKL2}.
	Note that from Theorem 4.2 in~\cite{doss2020optimal} we can obtain the following dimension-free bound
	$$\KL(f_\pi\|f_\eta)\lesssim e^{Ck^2} H^2(f_\pi, f_\eta)$$
	for some constant $C$, if we assume $\pi$ and $\eta$ are $k$-atomic distributions. But this bound depending exponentially on the number of components.
 
	\item \textbf{Minimax rate for estimating Gaussian mixtures:} Find the sharp rate of  
	$$
 R_{H^2, n}(\mathcal{P}_{com}(M)) =  \inf_{\hat{f}_n}\sup_{f\in \mathcal{P}_{com}(M)}\mathbb{E}_f[H^2(\hat{f}_n, f)].$$
 Thanks to the 
comparison inequality in Theorem \ref{thm-KL-d}, 
Corollary \ref{cor-1} reduces this problem to computing the local Hellinger entropy of the mixture class $\mathcal{P}_{com}(M)$. 
 The best known estimates for this in one dimension are
$$\log(1/\epsilon)\lesssim\log\calN_{loc, H}(\mathcal{P}_{com}(M), \epsilon)\lesssim \left(\frac{\log(1/\epsilon)}{\log\log(1/\epsilon)}\right)^{3/2}.$$
Here the lower bound is from Theorem 1.3 in~\cite{kim2014minimax}, which shows that $R_{H^2, n} = \Omega\left(\log n/n\right)$; the upper bound is from~\cite{NW21} by constructing a covering of the truncated moment space of the mixing distributions. The upper bound leads to an upper bound $O\left((\log n/\log\log n)^{1.5}/n\right)$ on the minimax risk, improving the $O((\log n)^2/n)$ result of \cite{kim2014minimax}.


	\item \textbf{Linear comparison between TV and Hellinger:} 
	It is well-known that $H^2 \lesssim \TV \lesssim H$ in general. Can we show that $\TV \asymp H$ for Gaussian mixtures? Specifically, for any two $\pi$ and $\eta$ supported on $[-M, M]$, can we show that there exists some constant $C=C(M)$ such that
	$$\TV(f_\pi, f_\eta)\ge C\cdot H(f_\pi, f_\eta).$$
	
	We notice that it is impossible to lower bound the $L_2$-distance $\|f_\pi- f_\eta\|_2$ linearly in $H(f_\pi, f_\eta)$,
	because~\cite{kim2014minimax} showed that for subgaussian mixing distributions, 
	the minimax squared $L_2$ risk for estimating the mixture density is at most $O(\sqrt{\log n}/n)$ and the squared Hellinger risk is at least $\Omega(\log n/n)$.
	Thus the best comparison between $L_2$ and $H$ will involve log factors. Similarly, the
	best known comparisons for $L_2$ and $\TV$, which we derive in
	Section~\ref{sec:other_dist}, also involve log-factors. It is an open problem to find
	tight log-factors in these comparisons of $L_2, H$ and $\TV$.
\end{enumerate}

\section{Proof of Theorem \ref{thm-KL-d} in one dimension}\label{sec-1d}
\par In this section, we provide a proof of Theorem $\KL\lesssim H^2$ for one-dimensional Gaussian mixtures where the mixing distribution is compactly supported. Similar proof techniques can also be applied in multiple dimensions; see Appendix \ref{app-1} for the proof of Theorem \ref{thm-KL-d} in general dimensions.
\begin{theorem}[One-dimensional version of Theorem \ref{thm-KL-d}]\label{thm-kl}
	Let $\pi$ and $\eta$ be supported on $B_2(M)$ in $\reals$ where $M\ge 2$. Then
$$\KL(f_\pi\|f_\eta)\le 1563 M^2 H^2(f_\pi, f_\eta).$$\end{theorem}

\par For simplicity we abbreviate $f_\pi(\cdot)$ and $f_\eta(\cdot)$ as $p(\cdot), q(\cdot)$. Then we have
$$\KL(p\|q) = \int_{-\infty}^\infty q(x)\cdot\frac{p(x)}{q(x)}\log\frac{p(x)}{q(x)}dx, \quad H^2(p, q) = \int_{-\infty}^\infty q(x)\cdot\left(\sqrt{\frac{p(x)}{q(x)}} - 1\right)^2dx.$$
We first state several lemmas. The first is a straightforward computation.
\begin{lemma}\label{lem-1}
	Let $p = \pi * \mathcal{N}(0, 1)$ where $\supp(\pi)\subset [-M, M]$. Then we have $\forall y\ge r\ge x\ge M$, 
	$$p(y)\exp\left(\frac{(y-M)^2 - (r-M)^2}{2}\right)\le p(r)\le p(x).$$
\end{lemma}

The following result bounds the grownth of the ``score function'' in the Gaussian mixture model.
\begin{lemma}\label{lem-2}
Let $p = \pi * \mathcal{N}(0, 1)$ where $\supp(\pi)\subset [-M, M]$. Then
	$$|\nabla \log p(r)|\le 3|r| + 4M, \quad \forall r\in\RR.$$
\end{lemma}
\begin{proof}
	By~\cite[Proposition 2]{polyanskiy2016wasserstein}, we have for all $r\in\RR$, $|\nabla \log p(r)|\le 3|r| + 4|\EE[X]|$, 	
	where $X\sim \pi$. Since $\pi$ is on $[-M, M]$, we have $|\EE[X]|\le M$.
\end{proof}
\begin{lemma}\label{lem-formula}
	For every $0\le t\le \exp\left(8M^2\right)$ with $M\ge 1$, we have
	$$t\log t - t + 1\le 9M^2\left(\sqrt{t} - 1\right)^2.$$
\end{lemma}
\begin{proof}
	We define
	$$g(t)\triangleq \frac{t\log t - t + 1}{(\sqrt{t}-1)^2}.$$
	Then we have $g'(t) = \frac{t-1 - \sqrt{t}\log t}{\sqrt{t}(\sqrt{t} - 1)^3}$. The numerator $h(t) = t - 1 - \sqrt{t}\log t$ within satisfies that $h'(t) = 1 - \frac{\log t}{2\sqrt{t}} - \frac{1}{\sqrt{t}} = \frac{\sqrt{t} - 1 - \log\sqrt{t}}{\sqrt{t}}\ge 0$. Hence for $0\le t\le 1$ we have $h(t)\le h(1) = 0$ and for $t\ge 1$ we have $h(t)\ge h(1) = 0$. Therefore, we have $g'(t) = \frac{h(t)}{\sqrt{t}(\sqrt{t} - 1)^3}\ge 0$ for all $t\ge 0$, which indicates that $g$ is non-decreasing on $t\ge 0$. Hence for $0\le t\le \exp\left(8M^2\right)$ and $M\ge 1$, we have
	$$g(t)\le g(\exp(8M^2))\le \frac{\exp(8M^2)\cdot 8M^2}{(\sqrt{\exp(8M^2)} - 1)^2} = \frac{8M^2}{(1 - \exp(-4M^2))^2}\le 9M^2,$$
	which indicates that
	$$t\log t - t + 1\le 9M^2\left(\sqrt{t} - 1\right)^2$$
\end{proof}
\begin{proof}[Proof of Theorem \ref{thm-kl}] It is easy to see that for every $x\in [-2M, 2M]$, we have
	$$\frac{1}{\sqrt{2\pi}}\exp\left(-8M^2\right)\le p(x), q(x)\le \frac{1}{\sqrt{2\pi}}.$$
	Let $r$ be the smallest positive number (possibly infinite) such that $\log \frac{p(r)}{q(r)}\ge 8M^2$, then we have $r\ge 2M$. Without loss of generality we assume $r < \infty$. (Otherwise $\int_r^\infty p(x)\log\frac{p(x)}{q(x)} - p(x) + q(x)dx = 0$ and there is nothing to prove.) Since $\log\frac{p(\cdot)}{q(\cdot)}$ is a continuous function, we have $\log\frac{p(r)}{q(r)} = 8M^2$. 
 According to Lemma \ref{lem-1}, we have for every $x\ge r$,
	$$p(x)\le p(r)\exp\left(-\frac{(x-M)^2 - (r-M)^2}{2}\right), \quad q(x)\le q(r)\exp\left(-\frac{(x-M)^2 - (r-M)^2}{2}\right)$$
	and according to Lemma \ref{lem-2} we have
	$$\left|\log\frac{p(x)}{q(x)}\right|\le \left|\log\frac{p(r)}{q(r)}\right| + \int_r^x (3|t|+4M)dt = 8M^2 + (x-r)(3x + 3r + 8M), \quad \forall x\ge r\ge 0.$$
	Therefore, we obtain that
	\begin{align*}
		&\quad \int_r^\infty p(x)\log\frac{p(x)}{q(x)} - p(x) + q(x)dx\le \int_r^\infty p(x)\log\frac{p(x)}{q(x)} + q(x)dx\\
		& \le p(r)\exp\left(\frac{(r-M)^2}{2}\right)\int_r^\infty (8M^2 + (x-r)(3x + 3r + 8M))\exp\left(-\frac{(x-M)^2}{2}\right)dx\\
		&\quad + q(r)\exp\left(\frac{(r-M)^2}{2}\right)\int_r^\infty \exp\left(-\frac{(x-M)^2}{2}\right)dx\\
		&\le p(r)\cdot\left(\frac{6}{(r-M)^3} + \frac{6r + 8M}{(r-M)^2} + \frac{8M^2}{r-M}\right) + \frac{q(r)}{r-M}\le \frac{36M^2}{r-M}p(r) + \frac{q(r)}{r-M}\\
		& \le \frac{37M^2}{r-M}p(r),
	\end{align*}
	where the last inequality uses the fact $M\ge 1$ and $\log \frac{p(r)}{q(r)} = 8M^2\ge 0$ hence $p(r)\ge q(r)$.
	
	\par Moreover, according to Lemma \ref{lem-2} we also notice that for $0\le x\le r$ we have $\left|\nabla \log\frac{p(x)}{q(x)}\right|\le 6x + 8M\le 6r + 8M$. 
 Hence noticing that $\log\frac{p(r)}{q(r)} = 8M^2$ and also $r\ge 2M$, we have for every $r - \frac{M^2}{r + M}\le x\le r$, 
	$$\log\frac{p(x)}{q(x)}\ge 8M^2 - \frac{M^2}{r+M}\cdot (6r + 8M)\ge M^2$$
	and also $p(x)\ge p(r)$ according to Lemma \ref{lem-1}. Therefore, noticing $M\ge 1$, we have
	$$H^2(p, q)\ge \int_{r - \frac{M^2}{r + M}}^rp(x)\cdot \left(\sqrt{\frac{q(x)}{p(x)}} - 1\right)^2dx\ge \frac{p(r)M^2}{(r+M)}\cdot \left(1 - \frac{1}{e^{M^2/2}}\right)^2\ge \frac{p(r)M^2}{7(r+M)}.$$
	Since $r\ge 2M$, we obtain that
	$$\int_r^\infty p(x)\log\frac{p(x)}{q(x)} - p(x) + q(x)dx\le 777 H^2(p, q).$$
	Similarly, if we let $s$ to be the largest negative number (possibly negative infinite) such that $\log \frac{p(s)}{q(s)}\ge 8M^2$, then we will also have
	$$\int_{-\infty}^s p(x)\log\frac{p(x)}{q(x)} - p(x) + q(x)dx\le 777H^2(p, q).$$
	\par Next we consider those $s\le x\le r$. For those $x$ we have $\log\frac{p(x)}{q(x)}\le  8M^2$. Hence according to Lemma \ref{lem-formula}, we have
	$$\frac{p(x)}{q(x)}\log\frac{p(x)}{q(x)} - \frac{p(x)}{q(x)} + 1\le 9M^2\left(\sqrt{\frac{p(x)}{q(x)}} - 1\right)^2.$$ 
	Therefore,
	\begin{align*}
	&\quad \int_s^r p(x)\log\frac{p(x)}{q(x)} - p(x) + q(x)dx = \int_s^r q(x)\cdot\left(\frac{p(x)}{q(x)}\log\frac{p(x)}{q(x)} - \frac{p(x)}{q(x)} + 1\right)dx\\
	& \leq 9M^2\int_s^r q(x)\cdot\left(\sqrt{\frac{p(x)}{q(x)}} - 1\right)^2dx\le 9M^2H^2(p, q).
	\end{align*}
	\par Overall, we have shown that
	$$\KL(p\|q)\le (777 + 777 + 9M^2)H^2(p, q)\le 1563M^2H^2(p, q),$$
	which finishes the proof of Theorem \ref{thm-kl}.
\end{proof}

\section{Proof of Theorem \ref{thm-KL-nd}} \label{sec-2}
\par First of all, we notice that \cite[Lemma 5]{haussler1997mutual} shows that for any $\delta, \lambda > 1$ such that
\begin{equation}\label{eq-condition}
0<\delta<\exp(-1/2)\quad\text{and}\quad\log\log(1/\delta)/\log(1/\delta)\le (\lambda - 1)/2,
\end{equation}
and any probability measures $\PP, \QQ, \SSS$ and $\QQ' = (1-\delta)\QQ + \delta \SSS$, we have
\begin{equation}
\KL(\PP\|\QQ')\le \frac{2\log(1/\delta)}{(1-\delta)^2}H^2(\PP, \QQ) + \frac{4\delta\log(1/\delta)}{(1-\delta)^2} + \delta^{\frac{\lambda - 1}{2}}\cdot\int_{\RR^d}\frac{(d\PP)^\lambda}{(d\SSS)^{\lambda - 1}}.
\label{eq:HO}
\end{equation}
Let $\lambda = 3$, then as long as $0 < \delta < 1/2$, \eqref{eq-condition} holds. Choose, $\PP = f_\pi$ and $\SSS = \QQ = f_\eta$, we get
\begin{equation}\label{eq-ho}\KL(f_\pi\|f_\eta)\le \frac{2\log(1/\delta)}{(1-\delta)^2}H^2(f_\pi, f_\eta) + \frac{4\delta\log(1/\delta)}{(1-\delta)^2} + \delta\cdot\int_{\RR^d}\frac{f_\pi(\mx)^3}{f_\eta(\mx)^2}d\mx.\end{equation}
Notice that the last term is an $f$-divergence with $f = x^3$, which is a convex function, hence $\int_{\RR^d}\frac{f_\pi(\mx)^3}{f_\eta(\mx)^2}d\mx$ is convex in $(f_\pi, f_\eta)$. Define the set $\mathscr{P}(M)$ 
of Gaussian mixtures with mixing distributions supported on the ball $B_2(M)$:
$$\mathscr{P}(M) = \{\pi * \mathcal{N}(0, I_d):\supp(\pi)\subset B_2(M)\}$$
which is a convex set. Hence the maximum value of $\int_{\RR^d}\frac{f_\pi(\mx)^3}{f_\eta(\mx)^2}d\mx$ where $f_\pi, f_\eta\in \mathscr{P}(M)$ is attained when $f_\pi, f_\eta$ are both at the boundary of $\mathscr{P}(M)$, i.e. $\exists \mathbf{u}, \mathbf{v}$ with $\|\mathbf{u}\|_2, \|\mathbf{v}\|_2\le M$ and we have $f_\pi = \delta_\mathbf{u} * \mathcal{N}(0, I_d), f_\eta = \delta_\mathbf{v} * \mathcal{N}(0, I_d)$. (This is because any $f_\pi\in\mathcal{P}(M)$ can be written as $\int_{B_2(M)}\pi(\mmu)f_{\delta_{\mmu}}d\mmu$.)
Therefore, we have
\begin{align*}
	\int_{\RR^d}\frac{f_\pi(\mx)^3}{f_\eta(\mx)^2}d\mx & \le \sup_{\mathbf{u}, \mathbf{v}: \|\mathbf{u}\|,  \|\mathbf{v}\|_2 \le M}\int_{\RR^d}\frac{\left(\frac{1}{\sqrt{2\pi}^d}\exp\left(-\frac{\|\mx + \mathbf{u}\|_2^2}{2}\right)\right)^3}{\left(\frac{1}{\sqrt{2\pi}^d}\exp\left(-\frac{\|\mx + \mathbf{v}\|_2^2}{2}\right)\right)^2}d\mx\\
	& = \sup_{\mathbf{u}, \mathbf{v}: \|\mathbf{u}\|,  \|\mathbf{v}\|_2 \le M}\frac{1}{\sqrt{2\pi}^d}\int_{\RR^d}\exp\left(-\frac{1}{2}\left(\mx^T\mx + 6\mx^T\mathbf{u} - 4\mx^T\mathbf{v} + 3\mathbf{u}^T\mathbf{u} - 2\mathbf{v}^T\mathbf{v}\right)\right)d\mx\\
	& = \sup_{\mathbf{u}, \mathbf{v}: \|\mathbf{u}\|,  \|\mathbf{v}\|_2 \le M}\frac{1}{\sqrt{2\pi}^d}\exp\left(3\|\mathbf{u} - \mathbf{v}\|_2^2\right)\cdot\int_{\RR^d}\exp\left(-\frac{\|\mx + 3\mathbf{u} - 2\mathbf{v}\|_2^2}{2}\right)d\mx\\
	& = \sup_{\mathbf{u}, \mathbf{v}: \|\mathbf{u}\|,  \|\mathbf{v}\|_2\le M}\exp\left(3\|\mathbf{u} - \mathbf{v}\|_2^2\right) = \exp\left(12M^2\right).
\end{align*}
Therefore, according to \eqref{eq-ho}, we have for any $\delta\in [0, 1/2]$,
$$\KL(f_\pi\|f_\eta)\le \frac{2\log(1/\delta)}{(1-\delta)^2}H^2(f_\pi, f_\eta) + \frac{4\delta\log(1/\delta)}{(1-\delta)^2} + \exp(12M^2)\delta.$$
Choosing $\delta = \exp(-12M^2)H^2(f_\pi, f_\eta)\in [0, 1/2]$ and noticing that $(1 - \delta)^2\ge \frac{1}{2}$, we get
\begin{align*}
	\KL(f_\pi\|f_\eta) &\le H^2(f_\pi, f_\eta) + 96M^2H^2(f_\pi, f_\eta) + 16H^2(f_\pi, f_\eta)\log\frac{1}{H^2(f_\pi, f_\eta)}\\
	&\le 97M^2H^2(f_\pi, f_\eta) +  16H^2(f_\pi, f_\eta)\log\frac{1}{H^2(f_\pi, f_\eta)}.
\end{align*}
This finishes the proof of Theorem \ref{thm-KL-nd}.

\section{Proof of Corollary \ref{cor-1}}\label{sec-proof-cor1}

For convenience, denote by $\tilde{R}_{H^2, n}$ the minimax squared Hellinger risk for improper density estimation, similar to $\tilde{R}_{KL, n}$. 
First, notice that for $\mathcal{P}\subset \mathcal{P}_{com}(M, d)$ or $\mathcal{P}\subset \mathcal{P}_{sub}(K, d)$
$$R_{H^2, n}(\mathcal{P})\le \tilde{R}_{H^2, n}(\mathcal{P}), \quad R_{KL, n}(\mathcal{P})\le \tilde{R}_{KL, n}(\mathcal{P}),$$
and also
$$R_{H^2, n}(\mathcal{P})\lesssim R_{KL, n}(\mathcal{P})$$
since $H^2(\PP, \QQ)\leq \KL(\PP\|\QQ)$ holds for all distributions $\PP$ and $\QQ$. Moreover, according to Theorems \ref{thm-KL-d} and \ref{thm-sub-ub}, for $\PP, \QQ\in\mathcal{P}$, we have $\KL(\PP\|\QQ)\lesssim H^2(\PP, \QQ)$, which indicates that
$$\tilde{R}_{KL, n}(\mathcal{P})\lesssim \tilde{R}_{H^2, n}(\mathcal{P}).$$
Therefore, we have
$$R_{H^2, n}(\mathcal{P})\lesssim R_{KL, n}(\mathcal{P})\le \tilde{R}_{KL, n}(\mathcal{P})\lesssim \tilde{R}_{H^2, n}(\mathcal{P}).$$
Next, we notice that
$$R_{H^2, n}(\mathcal{P}) =  \inf_{\hat{f}_n}\sup_{f\in\mathcal{P}}\mathbb{E}_f\left[H^2(\hat{f}_n, f)\right].$$
For one estimator $\hat{f}_n$, suppose $\tilde{f}_n$ is the projection of $\hat{f}_n$ into $\mathcal{P}$ under Hellinger distance (since $\mathcal{P}$ is convex, such projection always exists). Then for every $f\in \mathcal{P}$, we have
$$H(\tilde{f}_n, f)\le H(\tilde{f}, \hat{f}) + H(\hat{f}, f)\le 2H(\hat{f}, f),$$
where the last inequality uses the fact that $H(\tilde{f}, \hat{f})\le H(\hat{f}, f)$ due to projection. Here $\tilde{f}$ is a proper estimator. Therefore, we have
$$R_{H^2, n}(\mathcal{P}) =  \inf_{\hat{f}_n}\sup_{f\in\mathcal{P}}\mathbb{E}_f\left[H^2(\hat{f}_n, f)\right]\ge \frac{1}{4}\inf_{\hat{f}_n\in \mathcal{P}}\sup_{f\in\mathcal{P}}\mathbb{E}_f\left[H^2(\hat{f}_n, f)\right] = \frac{1}{4}\tilde{R}_{H^2, n}(\mathcal{P}).$$
Hence we have proved that
$$R_{H^2, n}(\mathcal{P})\lesssim R_{KL, n}(\mathcal{P})\le \tilde{R}_{KL, n}(\mathcal{P})\lesssim \tilde{R}_{H^2, n}(\mathcal{P})\lesssim R_{H^2, n}(\mathcal{P}),$$
so we have
$$R_{H^2, n}(\mathcal{P})\asymp R_{KL, n}(\mathcal{P})\asymp \tilde{R}_{KL, n}(\mathcal{P})\asymp \tilde{R}_{H^2, n}(\mathcal{P}).$$
Similarly, for sequential density estimation minimax risks, we can also show that
$$C_{H^2, n}(\mathcal{P})\asymp C_{KL, n}(\mathcal{P})\asymp \tilde{C}_{KL, n}(\mathcal{P})\asymp \tilde{C}_{H^2, n}(\mathcal{P}),$$
where $\tilde{C}_{KL, n}(\mathcal{P}), \tilde{C}_{H^2, n}(\mathcal{P})$ are the proper sequential density estimation minimax risks (where we restrict $\hat{f}_1, \cdots, \hat{f}_n$ to be in the class $\mathcal{P}$ in Definition \ref{def-sequential}. 
Therefore, to prove Corollary \ref{cor-1}, we only need to show:
\begin{align*}
	R_{H^2, n} &\asymp \inf_{\epsilon > 0} \epsilon^2 + \frac{1}{n}\log\mathcal{N}_{loc, H}(\mathcal{P}, \epsilon),\\
	C_{KL, n} &\asymp \inf_{\epsilon > 0} n\epsilon^2 + \log\mathcal{N}_{H}(\mathcal{P}, \epsilon).
\end{align*}
\par For the first inequality above, the upper bound part follows directly from the celebrated Le Cam-Birg\'{e} construction~\cite{lecam1973convergence, birge1983approximation, birge1986estimating}. The lower bound follows from applying Fano's inequality to a local Hellinger ball and the fact that 
$\KL(\PP\|\QQ)\asymp H^2(\PP, \QQ)$ for $\PP, \QQ\in\mathcal{P}$; see Corollary 33.2 in~\cite{polyanskiy2022information}.
\par The second inequality (on $C_{KL, n}$) follows directly from Lemma 6 and Lemma 7 in~\cite{haussler1997mutual} after noticing that the coefficient $b(\epsilon)$ in Lemma 7 of~\cite{haussler1997mutual} satisfies that 
$$b(\epsilon) = \sup\left\{\frac{\KL(\PP\|\QQ)}{H^2(\PP, \QQ)}: \PP, \QQ\in \mathcal{P}, H^2(\PP, \QQ)\le \epsilon\right\}\lesssim 1.$$

\acks{ZJ and YP were supported in part by the MIT-IBM Watson AI Lab and by the National Science
    Foundation under Grant No CCF-2131115. YW is supported in part by the NSF Grant CCF-1900507, an NSF CAREER award CCF-1651588, and an Alfred Sloan fellowship. The authors thank the anonymous referees for pointing out \cite[Theroem 5]{wong1995probability} and other helpful comments.}

\bibliography{reference}

\begin{thebibliography}{34}
\providecommand{\natexlab}[1]{#1}
\providecommand{\url}[1]{\texttt{#1}}
\expandafter\ifx\csname urlstyle\endcsname\relax
  \providecommand{\doi}[1]{doi: #1}\else
  \providecommand{\doi}{doi: \begingroup \urlstyle{rm}\Url}\fi

\bibitem[Ashtiani et~al.(2020)Ashtiani, Ben-David, Harvey, Liaw, Mehrabian, and
  Plan]{ashtiani2020near}
Hassan Ashtiani, Shai Ben-David, Nicholas~JA Harvey, Christopher Liaw, Abbas
  Mehrabian, and Yaniv Plan.
\newblock Near-optimal sample complexity bounds for robust learning of gaussian
  mixtures via compression schemes.
\newblock \emph{Journal of the ACM (JACM)}, 67\penalty0 (6):\penalty0 1--42,
  2020.

\bibitem[Bandeira et~al.(2020)Bandeira, Niles-Weed, and
  Rigollet]{bandeira2020optimal}
Afonso~S Bandeira, Jonathan Niles-Weed, and Philippe Rigollet.
\newblock Optimal rates of estimation for multi-reference alignment.
\newblock \emph{Mathematical Statistics and Learning}, 2\penalty0 (1):\penalty0
  25--75, 2020.

\bibitem[Birg\'e(1983)]{birge1983approximation}
Lucien Birg\'e.
\newblock Approximation dans les espaces m\'etriques et th\'eorie de
  l'estimation.
\newblock \emph{Z. Wahrsch. Verw. Gebiete}, 65\penalty0 (2):\penalty0 181--237,
  1983.
\newblock ISSN 0044-3719.
\newblock \doi{10.1007/BF00532480}.
\newblock URL
  \url{http://dx.doi.org.offcampus.lib.washington.edu/10.1007/BF00532480}.

\bibitem[Birg{\'e}(1986)]{birge1986estimating}
Lucien Birg{\'e}.
\newblock On estimating a density using hellinger distance and some other
  strange facts.
\newblock \emph{Probability theory and related fields}, 71\penalty0
  (2):\penalty0 271--291, 1986.

\bibitem[Birg{\'e} and Massart(1998)]{birge1998minimum}
Lucien Birg{\'e} and Pascal Massart.
\newblock Minimum contrast estimators on sieves: exponential bounds and rates
  of convergence.
\newblock \emph{Bernoulli}, 4\penalty0 (3):\penalty0 329--375, 1998.
\newblock ISSN 1350-7265.
\newblock \doi{10.2307/3318720}.
\newblock URL \url{http://dx.doi.org/10.2307/3318720}.

\bibitem[Block et~al.(2022)Block, Jia, Polyanskiy, and Rakhlin]{block2022rate}
Adam Block, Zeyu Jia, Yury Polyanskiy, and Alexander Rakhlin.
\newblock Rate of convergence of the smoothed empirical wasserstein distance.
\newblock \emph{arXiv preprint arXiv:2205.02128}, 2022.

\bibitem[Chen and Niles-Weed(2021)]{chen2021asymptotics}
Hong-Bin Chen and Jonathan Niles-Weed.
\newblock Asymptotics of smoothed wasserstein distances.
\newblock \emph{Potential Analysis}, pages 1--25, 2021.

\bibitem[Doss et~al.(2020)Doss, Wu, Yang, and Zhou]{doss2020optimal}
Natalie Doss, Yihong Wu, Pengkun Yang, and Harrison~H Zhou.
\newblock Optimal estimation of high-dimensional location gaussian mixtures.
\newblock \emph{arXiv preprint arXiv:2002.05818}, 2020.

\bibitem[Fan et~al.(2021)Fan, Sun, Wang, and Wu]{fan2020likelihood}
Zhou Fan, Yi~Sun, Tianhao Wang, and Yihong Wu.
\newblock Likelihood landscape and maximum likelihood estimation for the
  discrete orbit recovery model.
\newblock \emph{Communications on Pure and Applied Mathematics}, 2021.

\bibitem[Genovese and Wasserman(2000)]{genovese2000rates}
Christopher~R Genovese and Larry Wasserman.
\newblock Rates of convergence for the gaussian mixture sieve.
\newblock \emph{The Annals of Statistics}, 28\penalty0 (4):\penalty0
  1105--1127, 2000.

\bibitem[Ghosal and van~der Vaart(2007)]{ghosal2007posterior}
Subhashis Ghosal and Aad van~der Vaart.
\newblock Posterior convergence rates of {D}irichlet mixtures at smooth
  densities.
\newblock \emph{Ann. Statist.}, 35\penalty0 (2):\penalty0 697--723, 2007.
\newblock ISSN 0090-5364.
\newblock \doi{10.1214/009053606000001271}.
\newblock URL \url{http://dx.doi.org/10.1214/009053606000001271}.

\bibitem[Ghosal and van~der Vaart(2001)]{ghosal2001entropies}
Subhashis Ghosal and Aad~W. van~der Vaart.
\newblock Entropies and rates of convergence for maximum likelihood and {B}ayes
  estimation for mixtures of normal densities.
\newblock \emph{Ann. Statist.}, 29\penalty0 (5):\penalty0 1233--1263, 2001.
\newblock ISSN 0090-5364.
\newblock \doi{10.1214/aos/1013203453}.
\newblock URL \url{https://doi.org/10.1214/aos/1013203453}.

\bibitem[Gr{\"u}nwald and Mehta(2020)]{grunwald2020fast}
Peter~D Gr{\"u}nwald and Nishant~A Mehta.
\newblock Fast rates for general unbounded loss functions: from erm to
  generalized bayes.
\newblock \emph{The Journal of Machine Learning Research}, 21\penalty0
  (1):\penalty0 2040--2119, 2020.

\bibitem[Haussler and Opper(1997)]{haussler1997mutual}
David Haussler and Manfred Opper.
\newblock Mutual information, metric entropy and cumulative relative entropy
  risk.
\newblock \emph{The Annals of Statistics}, 25\penalty0 (6):\penalty0
  2451--2492, 1997.

\bibitem[Ibragimov(2001)]{ibragimov2001estimation}
Ildar Ibragimov.
\newblock Estimation of analytic functions.
\newblock \emph{Lecture Notes-Monograph Series}, pages 359--383, 2001.

\bibitem[Kim(2014)]{kim2014minimax}
Arlene~KH Kim.
\newblock Minimax bounds for estimation of normal mixtures.
\newblock \emph{bernoulli}, 20\penalty0 (4):\penalty0 1802--1818, 2014.

\bibitem[Kim and Guntuboyina(2022)]{kim2022minimax}
Arlene~KH Kim and Adityanand Guntuboyina.
\newblock Minimax bounds for estimating multivariate gaussian location
  mixtures.
\newblock \emph{Electronic Journal of Statistics}, 16\penalty0 (1):\penalty0
  1461--1484, 2022.

\bibitem[Le~Cam(1973)]{lecam1973convergence}
Lucien Le~Cam.
\newblock Convergence of estimates under dimensionality restrictions.
\newblock \emph{Ann. Statist.}, 1:\penalty0 38--53, 1973.
\newblock ISSN 0090-5364.
\newblock URL
  \url{http://links.jstor.org.offcampus.lib.washington.edu/sici?sici=0090-5364(197301)1:1<38:COEUDR>2.0.CO;2-V&origin=MSN}.

\bibitem[Mourtada(2023)]{jaonad2023coding}
Jaouad Mourtada.
\newblock Coding convex bodies under gaussian noise, and the wills functional.
\newblock \emph{Draft}, 2023.

\bibitem[Neykov(2022)]{neykov2022minimax}
Matey Neykov.
\newblock On the minimax rate of the gaussian sequence model under bounded
  convex constraints.
\newblock \emph{IEEE Transactions on Information Theory}, 2022.

\bibitem[Nie and Wu(2021)]{NW21}
Yutong Nie and Yihong Wu.
\newblock Improved rates for estimating gaussian mixtures.
\newblock \emph{Draft}, Dec 2021.

\bibitem[Polyanskiy and Wu(2016)]{polyanskiy2016wasserstein}
Yury Polyanskiy and Yihong Wu.
\newblock Wasserstein continuity of entropy and outer bounds for interference
  channels.
\newblock \emph{IEEE Transactions on Information Theory}, 62\penalty0
  (7):\penalty0 3992--4002, 2016.

\bibitem[Polyanskiy and Wu(2021)]{polyanskiy2021sharp}
Yury Polyanskiy and Yihong Wu.
\newblock Sharp regret bounds for empirical bayes and compound decision
  problems.
\newblock \emph{arXiv preprint arXiv:2109.03943}, 2021.

\bibitem[Polyanskiy and Wu(2022+)]{polyanskiy2022information}
Yury Polyanskiy and Yihong Wu.
\newblock Information theory: From coding to learning.
\newblock \emph{Cambridge University Press}, 2022+.
\newblock URL
  \url{https://people.lids.mit.edu/yp/homepage/data/itbook-export.pdf}.

\bibitem[Saha and Guntuboyina(2020)]{saha2020nonparametric}
Sujayam Saha and Adityanand Guntuboyina.
\newblock On the nonparametric maximum likelihood estimator for gaussian
  location mixture densities with application to gaussian denoising.
\newblock \emph{The Annals of Statistics}, 48\penalty0 (2):\penalty0 738--762,
  2020.

\bibitem[Sason and Verd{\'u}(2016)]{sason2016f}
Igal Sason and Sergio Verd{\'u}.
\newblock $ f $-divergence inequalities.
\newblock \emph{IEEE Transactions on Information Theory}, 62\penalty0
  (11):\penalty0 5973--6006, 2016.

\bibitem[van~de Geer(1993)]{van1993hellinger}
Sara van~de Geer.
\newblock Hellinger-consistency of certain nonparametric maximum likelihood
  estimators.
\newblock \emph{Ann. Statist.}, 21\penalty0 (1):\penalty0 14--44, 1993.
\newblock ISSN 0090-5364.
\newblock \doi{10.1214/aos/1176349013}.
\newblock URL \url{https://doi.org/10.1214/aos/1176349013}.

\bibitem[van~de Geer(1996)]{van1996rates}
Sara van~de Geer.
\newblock Rates of convergence for the maximum likelihood estimator in mixture
  models.
\newblock \emph{Journal of Nonparametric Statistics}, 6\penalty0 (4):\penalty0
  293--310, 1996.

\bibitem[Wong and Shen(1995)]{wong1995probability}
Wing~Hung Wong and Xiaotong Shen.
\newblock Probability inequalities for likelihood ratios and convergence rates
  of sieve {MLE}s.
\newblock \emph{Ann. Statist.}, 23\penalty0 (2):\penalty0 339--362, 1995.
\newblock ISSN 0090-5364.
\newblock \doi{10.1214/aos/1176324524}.
\newblock URL \url{http://dx.doi.org/10.1214/aos/1176324524}.

\bibitem[Wu and Yang(2020{\natexlab{a}})]{wu2020optimal}
Yihong Wu and Pengkun Yang.
\newblock Optimal estimation of gaussian mixtures via denoised method of
  moments.
\newblock \emph{The Annals of Statistics}, 48\penalty0 (4):\penalty0
  1981--2007, 2020{\natexlab{a}}.

\bibitem[Wu and Yang(2020{\natexlab{b}})]{wu2020polynomial}
Yihong Wu and Pengkun Yang.
\newblock Polynomial methods in statistical inference: Theory and practice.
\newblock \emph{Foundations and Trends® in Communications and Information
  Theory}, 17\penalty0 (4):\penalty0 402--586, 2020{\natexlab{b}}.
\newblock ISSN 1567-2190.
\newblock \doi{10.1561/0100000095}.
\newblock URL \url{http://dx.doi.org/10.1561/0100000095}.

\bibitem[Yang and Barron(1999)]{yang1999information}
Yuhong Yang and Andrew Barron.
\newblock Information-theoretic determination of minimax rates of convergence.
\newblock \emph{Ann. Statist.}, 27\penalty0 (5):\penalty0 1564--1599, 1999.
\newblock ISSN 0090-5364.
\newblock \doi{10.1214/aos/1017939142}.
\newblock URL \url{http://dx.doi.org/10.1214/aos/1017939142}.

\bibitem[Yatracos(1985)]{yatracos1985rates}
Yannis~G Yatracos.
\newblock Rates of convergence of minimum distance estimators and kolmogorov's
  entropy.
\newblock \emph{The Annals of Statistics}, 13\penalty0 (2):\penalty0 768--774,
  1985.

\bibitem[Zhang(2009)]{zhang2009generalized}
Cun-Hui Zhang.
\newblock Generalized maximum likelihood estimation of normal mixture
  densities.
\newblock \emph{Statistica Sinica}, pages 1297--1318, 2009.

\end{thebibliography}


\newpage
\appendix

\section{Proof of Theorem \ref{thm-KL-d}}\label{app-1}

\par Without loss of generality, we assume $d\le M^2$ (otherwise we use $\sqrt{d}\ge M$ to replace $M$, and since $\supp(\pi), \supp(\eta)\subset B_2(M)\subset B_2(\sqrt{d})$, the results still hold). For simplicity we abbreviate $f_\pi(\cdot)$ and $f_\eta(\cdot)$ as $p(\cdot), q(\cdot)$. Then we can write
\begin{equation}\label{eq-kl-hellinger}
\KL(p\|q) = \int_{\mathbb{R}^d} p(\mx)\log\frac{p(\mx)}{q(\mx)}d\mx, \quad H^2(p, q) = \int_{\mathbb{R}^d} p(\mx)\cdot\left(\sqrt{\frac{q(\mx)}{p(\mx)}} - 1\right)^2d\mx.
\end{equation}
\par Denoting by $\Omega$  the unit sphere in $\RR^d$, 
each $\mx$ in $\mathbb{R}^d$ can be written as $\mx(r, \omega)$, with $r = \|\mx\|_2$ and $\omega$ to be the vector parallel to $\mx$ in $\Omega$.  

\par To prove Theorem \ref{thm-KL-d}, we need the following lemmas.
\begin{lemma}\label{lem-eq}
	Suppose $p = \pi * \mathcal{N}(0, I_d)$, where $\supp(\pi)\subset B_2(M)$. Then for any $\omega\in \Omega$, we have:
	\begin{enumerate}
		\item $\forall r'\in [M, r]$, we have $p(\mx(r', \omega))\ge p(\mx(r, \omega))$.
		\item $\forall r'\ge r\ge M$, we have $p(\mx(r', \omega))\le p(\mx(r, \omega))\exp\left(-\frac{(r'-M)^2 - (r-M)^2}{2}\right)$
	\end{enumerate}
\end{lemma}
\begin{proof} We can write
	$$p(\mx(r', \omega)) = \int_{B_2(M)} \varphi(\mx(r', \omega) - \mmu)\pi(\mmu)d\mmu,\qquad p(\mx(r, \omega)) = \int_{B_2(M)} \varphi(\mx(r, \omega) - \mmu)\pi(\mmu)d\mmu,$$
	where we use $\pi(\cdot)$ to denote the density distribution of $\pi$ (which can be a generalized function), and $\varphi(\cdot)$ to denote the density distribution of $\mathcal{N}(0, I_d)$. To prove this lemma, we only need to verify the following two inequalities:
	\begin{enumerate}
		\item For any $\forall r'\in [M, r]$ and any $u\in B_2(M)$, we have $\varphi(\mx(r', \omega) - \mmu)\ge \varphi(\mx(r, \omega) - \mmu)$;
		\item For any $\forall r'\ge r\ge M$ and any $u\in B_2(M)$, we have $\varphi(\mx(r', \omega) - \mmu)\le \varphi(\mx(r, \omega) - \mmu)\exp\left(-\frac{(r'-M)^2 - (r-M)^2}{2}\right)$.
	\end{enumerate}
	Without loss of generality, we assume $\omega = (1, 0, \cdots, 0)$. Then for any $\mmu = (u_1, u_2, \cdots, u_d)$, we have
	\begin{align*}
		\varphi(\mx(r, \omega) - \mmu) & = \frac{1}{\sqrt{2\pi}^d}\exp\left(-\frac{(r - u_1)^2 + \sum_{i=2}^d u_i^2}{2}\right)\\
		\varphi(\mx(r', \omega) - \mmu) & = \frac{1}{\sqrt{2\pi}^d}\exp\left(-\frac{(r' - u_1)^2 + \sum_{i=2}^d u_i^2}{2}\right).
	\end{align*}
	When $M\le r'\le r$, it is easy to see that $|r - u_1|\ge |r' - u_1|$ for any $|u_1|\le M$. The first inequality is verified. As for the second inequality, since $|u_1|\le M\le r\le r'$, we have $(r'-u_1)^2 - (r-u_1)^2\ge (r'-M)^2 - (r-M)^2$, which indicates that
	$$-\frac{(r' - u_1)^2 + \sum_{i=2}^d u_i^2}{2}\le -\frac{(r - u_1)^2 + \sum_{i=2}^d u_i^2}{2} - \frac{(r'-M)^2 - (r-M)^2}{2}.$$
\end{proof}

\begin{lemma}\label{lem-eq2}
	Suppose $p = \pi * \mathcal{N}(0, I_d)$ where $\supp(\pi)\subset B_2(M)$. Then we have
	$$\|\nabla \log p(\mx)\|_2\le 3\|\mx\|_2 + 4M, \quad \forall \mx\in\RR^d.$$
\end{lemma}
\begin{proof}
	According to Proposition 2 in~\cite{polyanskiy2016wasserstein}, we have $\forall \mx\in\RR^d$,
	$$\|\nabla \log p(\mx)\|_2\le 3\|\mx\|_2 + 4\|\EE[X]\|_2,$$
	where $X\sim \pi$. Since the support of $\pi$ is a subset of $B_2(M)$, we have $\|\EE[X]\|_2\le M$.
\end{proof}

\begin{proof}[Proof of Theorem \ref{thm-KL-d}]
	According to \eqref{eq-kl-hellinger}, we have
	\begin{equation}\label{eq-polar}
		\begin{aligned}
		\KL(p\|q) & = \int_\Omega\int_0^\infty r^{d-1}p(\mx(r, \omega))\log\frac{p(\mx(r, \omega))}{q(\mx(r, \omega))}drd\omega\\
		& = \int_\Omega\int_0^\infty r^{d-1}\left(p(\mx(r, \omega))\log\frac{p(\mx(r, \omega))}{q(\mx(r, \omega))} - p(\mx(r, \omega)) + q(\mx(r, \omega))\right)drd\omega\\
		H^2(p, q) & = \int_\Omega\int_0^\infty r^{d-1}p(\mx(r, \omega))\left(\sqrt{\frac{q(\mx(r, \omega))}{p(\mx(r, \omega))}} - 1\right)^2drd\omega
	\end{aligned}\end{equation}
	For every $\omega\in\Omega$, we define $r_\omega$ as
	$$r_\omega\triangleq \inf\left\{r:\log \frac{p(\mx(r_\omega, \omega))}{q(\mx(r_\omega, \omega))}\ge 8M^2\right\}.$$
	Notice that for any $r\le 2M$ and $\omega\in\Omega$, we have
	\begin{align*}
		p(\mx(r, \omega)) & = \int_{B_2(M)}\pi(\mmu)\varphi(\mx(r, \omega), \mmu)d\mmu\le \frac{1}{\sqrt{2\pi}^d}\\
		q(\mx(r, \omega)) & = \int_{B_2(M)}\eta(\mmu)\varphi(\mx(r, \omega), \mmu)d\mmu\ge \frac{1}{\sqrt{2\pi}^d}\exp\left(-\frac{(2M + M)^2}{2}\right),
	\end{align*}
	which indicates that
	$$\log\frac{p(\mx(r, \omega))}{q(\mx(r, \omega)}\le \frac{9M^2}{2} < 8M^2.$$
	Hence for every $\omega\in \Omega$, we all have $r_\omega\ge 2M$. And if $r_\omega\neq \infty$, we have that $\log \frac{p(\mx(r_\omega, \omega))}{q(\mx(r_\omega, \omega))} = 8M^2$. According to Lemma \ref{lem-eq} and Lemma \ref{lem-eq2}, we know that for every $r\ge r_\omega$, we all have
	\begin{align*}
		p(\mx(r, \omega)) & \le p(\mx(r_\omega, \omega))\exp\left(-\frac{(r-M)^2 - (r_\omega-M)^2}{2}\right),\\
		\log \frac{p(\mx(r, \omega))}{q(\mx(r, \omega))} & \le \log \frac{p(\mx(r_\omega, \omega))}{q(\mx(r_\omega, \omega))} + (r - r_\omega)(3r + 3r_\omega + 8M)\\
		&\le 8M^2 + (r - r_\omega)(3r + 3r_\omega + 8M).
	\end{align*}
	Therefore, we obtain that
	\begin{equation}\label{eq-pr}\begin{aligned}
		&\quad \int_{r_\omega}^\infty r^{d-1}p(\mx(r, \omega))\log \frac{p(\mx(r, \omega))}{q(\mx(r, \omega))}dr\\
		&\le p(\mx(r_\omega, \omega))\int_{r_\omega}^\infty r^{d-1}(8M^2 + (r-r_\omega)(3r + 3r_\omega + 8M))\exp\left(-\frac{(r-M)^2 - (r_\omega-M)^2}{2}\right)dr.
	\end{aligned}\end{equation}
	We adopt the changes of variables from $r$ to $t = r - r_\omega$, and obtain that
	\begin{equation}\label{eq-first}
	\begin{aligned}
		&\quad 8M^2\int_{r_\omega}^\infty r^{d-1}\exp\left(-\frac{(r-M)^2 - (r_\omega-M)^2}{2}\right)dr\\
		& = 8M^2 r_\omega^{d-1}\int_0^\infty\exp\left((d-1)\log\frac{t + r_\omega}{r_\omega} - \frac{t^2}{2} - (r_\omega - M)t\right)dt
	\end{aligned}\end{equation}
	and
	\begin{equation}\label{eq-second}
	\begin{aligned}
		&\quad \int_{r_\omega}^\infty r^{d-1}(r-r_\omega)(3r + 3r_\omega + 8M)\exp\left(-\frac{(r-M)^2 - (r_\omega-M)^2}{2}\right)dr\\
		& = r_\omega^{d-1}\int_0^\infty\exp\left((d-1)\log\frac{t + r_\omega}{r_\omega} + \log t + \log (3t + 6r_\omega + 8M) - \frac{t^2}{2} - (r_\omega - M)t\right)dt.
	\end{aligned}\end{equation}
	\par We first prove the upper bound \eqref{eq-first}. We define
	$$f(t) \triangleq (d-1)\log\frac{t + r_\omega}{r_\omega} - \frac{1}{2}(r_\omega - M)t.$$
	Since 
	$$d - 1 < d\le M^2 = \frac{1}{2}\cdot (2M)\cdot (2M - M)\le \frac{1}{2}r_\omega(r_\omega - M),$$
	the first-order derivative of $f$ satisfies that 
	$$f'(t) = \frac{d-1}{t + r_\omega} - \frac{1}{2}(r_\omega - M)\le \frac{d-1}{r_\omega} - \frac{1}{2}(r_\omega - M)\le 0, \quad \forall t\ge 0.$$
	Therefore we have
	$$(d-1)\log\frac{t + r_\omega}{r_\omega} - \frac{1}{2}(r_\omega - M)t = f(t)\le f(0) = 0,\quad \forall t\ge 0.$$
	This directly indicates the following upper bound on \eqref{eq-first}.
	\begin{equation}\label{eq-exp}\begin{aligned}
		&\quad \frac{1}{r_\omega^{d-1}}\int_{r_\omega}^\infty r^{d-1}\exp\left(-\frac{(r-M)^2 - (r_\omega-M)^2}{2}\right)dr \le \int_0^\infty\exp\left(-\frac{t^2}{2} - \frac{1}{2}(r_\omega - M)t\right)dt\\
		&\le \int_0^\infty\exp\left(-\frac{1}{2}(r_\omega - M)t\right)dt = \frac{2}{r_\omega - M}
	\end{aligned}\end{equation}
	\par Next we prove the upper bound \eqref{eq-second}. We define
	$$g(t)\triangleq \log t + \log (3t + 6r_\omega + 8M) - \frac{1}{3}(r_\omega - M)t.$$
	Then we have
	$$g'(t) = \frac{1}{t} + \frac{1}{3t + 6r_\omega + 8M} - \frac{r_\omega - M}{3},$$
	which has a single root $t_0$ on $(0, \infty)$, which is also the maximum of $g(t)$ for $t\ge 0$. Further notice 
	$$0 = g'(t_0) = \frac{1}{t_0} + \frac{1}{3t_0 + 6r_\omega + 8M} - \frac{r_\omega - M}{3}\le \frac{4}{3t_0} - \frac{r_\omega - M}{3}.$$
	Hence we get $t_0\le \frac{4}{r_\omega - M}\le \frac{4}{M}$, and
	$$3t_0 + 6r_\omega + 8M\le \frac{12}{M} + 6r_\omega + 8M\le 6r_\omega + 20M\le 32(r_\omega - M).$$
	This gives the following upper bound on $g(t)$ for $t\ge 0$:
	$$g(t)\le g(t_0)\le \log t_0(3t_0 + 6r_\omega + 8M) - \frac{4}{3}\le \log 128 - \frac{4}{3},$$
	Combining this result with our upper bound on the function $f$, we get
	\begin{align*}
		&\quad (d-1)\log\frac{t + r_\omega}{r_\omega} + \log t + \log (3t + 6r_\omega + 8M) - \frac{t^2}{2} - (r_\omega - M)t\\
		& = f(t) + g(t) - \frac{t^2}{2} - \frac{(r_\omega - M)t}{6}\le \log 128 - \frac{4}{3} - \frac{(r_\omega - M)t}{6}.
	\end{align*}
	This directly indicates the following upper bound on \eqref{eq-second}.
	\begin{align*}
		&\quad \int_{r_\omega}^\infty r^{d-1}(r-r_\omega)(3r + 3r_\omega + 8M)\exp\left(-\frac{(r-M)^2 - (r_\omega-M)^2}{2}\right)dr\\
		&\le r_\omega^{d-1}\int_0^\infty \exp\left(\log 128 - \frac{4}{3} - \frac{(r_\omega - M)t}{6}\right)dt = 128r_\omega^{d-1}\cdot \frac{6e^{-4/3}}{r_\omega - M}\\
		& \le \frac{210r_\omega^{d-1}}{r_\omega - M}\le \frac{210M^2r_\omega^{d-1}}{r_\omega - M}.
	\end{align*}
	We combine these two upper bounds together. According to \eqref{eq-pr} we obtain that
	\begin{align*}
		&\quad \int_{r_\omega}^\infty r^{d-1}p(\mx(r, \omega))\log \frac{p(\mx(r, \omega))}{q(\mx(r, \omega))}dr\le \frac{212M^2r_\omega^{d-1}}{r_\omega - M} p(\mx(r_\omega, \omega)),
	\end{align*}
	\par Similarly, we can also obtain bound on $\int_{r_\omega}^\infty r^{d-1}q(\mx(r, \omega))dr$: According to Lemma \ref{lem-eq} we obtain that
	$$q(\mx(r, \omega)) \le q(\mx(r_\omega, \omega))\exp\left(-\frac{(r-M)^2 - (r_\omega-M)^2}{2}\right).$$
	According to \eqref{eq-exp} we have
	$$\int_{r_\omega}^\infty r^{d-1}q(\mx(r, \omega))dr\le q(\mx(r_\omega, \omega))\int_{r_\omega}^\infty r^{d-1}\exp\left(-\frac{(r-M)^2 - (r_\omega-M)^2}{2}\right)dr\le \frac{2r_\omega^{d-1}}{r_\omega - M}q(\mx(r_\omega, \omega)).$$
	We further notice that 
	$$\log \frac{p(\mx(r_\omega, \omega))}{q(\mx(r_\omega, \omega))} = 8M^2\ge 0,$$
	which indicates that $q(\mx(r_\omega, \omega))\le p(\mx(r_\omega, \omega))$. Therefore, since $M\ge 1$, we have
	$$\int_{r_\omega}^\infty r^{d-1}q(\mx(r, \omega))dr\le \frac{2M^2r_\omega^{d-1}}{r_\omega - M}p(\mx(r_\omega, \omega)).$$
	\par Therefore, we have the following upper bound 
	\begin{align*}
	&\quad \int_\Omega\int_{r_\omega}^\infty r^{d-1}\left(p(\mx(r, \omega))\log\frac{p(\mx(r, \omega))}{q(\mx(r, \omega))} - p(\mx(r, \omega)) + q(\mx(r, \omega))\right)drd\omega\\
	& \le \int_\Omega\int_{r_\omega}^\infty r^{d-1}p(\mx(r, \omega))\log\frac{p(\mx(r, \omega))}{q(\mx(r, \omega))}drd\omega + \int_\Omega\int_{r_\omega}^\infty r^{d-1}q(\mx(r, \omega))drd\omega\\
	& \le \int_\Omega\frac{(212 + 2)M^2r_\omega^{d-1}p(\mx(r_\omega, \omega))}{r_\omega - M}d\omega = \int_\Omega\frac{214M^2r_\omega^{d-1}p(\mx(r_\omega, \omega))}{r_\omega - M}d\omega.
	\end{align*}

	\par Next, according to Lemma \ref{lem-eq2}, for $\forall \omega\in \Omega$ and $r\in [0, r_\omega]$ we have
	\begin{align*}
		\left\|\nabla \log\frac{p(\mx(r, \omega))}{q(\mx(r, \omega))}\right\|_2 & = \|\nabla \log p(\mx(r, \omega))\|_2 + \|\nabla \log q(\mx(r, \omega))\|_2\\
		& \le 6r + 8M\le 6r_\omega + 8M\le  8r_\omega + 8M.
	\end{align*}
	Notice that we also have $\log\frac{p(\mx(r_\omega, \omega))}{q(\mx(r_\omega, \omega))} = 8M^2$. Hence for $\forall r_\omega - \frac{1}{r_\omega + M}\le r\le r_\omega$, 
	$$\log\frac{p(\mx(r, \omega))}{q(\mx(r, \omega))}\ge 8M^2 - (8r_\omega + 8M)\cdot \frac{1}{r_\omega + M} = 8M^2 - 8\ge 2,$$
	which indicates that
	$$\left(\sqrt{\frac{q(\mx(r, \omega))}{p(\mx(r, \omega))}} - 1\right)^2\ge \left(\frac{1}{e^2} - 1\right)^2\ge \frac{1}{2}.$$
	We further notice that $r_\omega\ge 2M\ge M$ and $d-1 < d\le M^2$. Hence for $\forall r_\omega - \frac{1}{r_\omega + M}\le r\le r_\omega$, 
	$$r^{d-1}\ge r_\omega^{d-1}\cdot \left(1 - \frac{1}{r_\omega(r_\omega + M)}\right)^{d-1}\ge r_\omega^{d-1}\cdot \left(1 - \frac{d}{r_\omega(r_\omega + M)}\right)\ge r_\omega^{d-1}\left(1 - \frac{d-1}{2M^2}\right)\ge \frac{1}{2}r_\omega^{d-1}.$$
	After noticing that $p(\mx(r, \omega))\ge p(\mx(r_\omega, \omega))$ according to Lemma \ref{lem-eq}, we obtain
	\begin{align*}
		&\quad \int_{r_\omega - \frac{1}{r_\omega + M}}^{r_\omega}r^{d-1}p(\mx(r, \omega))\cdot \left(\sqrt{\frac{q(\mx(r, \omega))}{p(\mx(r, \omega))}} - 1\right)^2dr\\
		& \ge \frac{1}{r_\omega + M}\cdot \min_{r_\omega - \frac{1}{r_\omega + M}\le r\le r_\omega}r^{d-1}p(\mx(r, \omega))\cdot \left(\sqrt{\frac{q(\mx(r, \omega))}{p(\mx(r, \omega))}} - 1\right)^2\\
		&\ge \frac{1}{r_\omega + M}\cdot \frac{1}{2}r_\omega^{d-1}\cdot \frac{1}{2}p(\mx(r_\omega, \omega))\ge \frac{r_\omega^{d-1}p(\mx(r_\omega, \omega))}{12(r_\omega - M)}\\
		& \ge \frac{1}{2568M^2}\int_{r_\omega}^\infty r^{d-1}\left(p(\mx(r, \omega))\log\frac{p(\mx(r, \omega))}{q(\mx(r, \omega))} - p(\mx(r, \omega)) + q(\mx(r, \omega))\right)dr.
	\end{align*}
	Therefore, according to \eqref{eq-polar}, we have
	\begin{align*}
		H^2(p, q) & = \int_\Omega\int_0^\infty r^{d-1}p(\mx(r, \omega))\left(\sqrt{\frac{q(\mx(r, \omega))}{p(\mx(r, \omega))}} - 1\right)^2drd\omega\\
		& \ge \int_\Omega\int_{r_\omega - \frac{1}{r_\omega + M}}^{r_\omega} r^{d-1}p(\mx(r, \omega))\left(\sqrt{\frac{q(\mx(r, \omega))}{p(\mx(r, \omega))}} - 1\right)^2drd\omega\\
		& \ge \frac{1}{2568 M^2}\int_\Omega\int_{r_\omega}^\infty r^{d-1}\left(p(\mx(r, \omega))\log\frac{p(\mx(r, \omega))}{q(\mx(r, \omega))} - p(\mx(r, \omega)) + q(\mx(r, \omega))\right)drd\omega.
	\end{align*}

	\par Next we consider those $\mx(r, \omega)$ with $0\le r\le r_\omega$. According to the definition of $r_\omega$, we know that for any such $r$, we have 
	$$\log\frac{p(\mx(r, \omega))}{q(\mx(r, \omega))}\le  8M^2.$$
	Hence from Lemma \ref{lem-formula}, we have for any $\omega\in\Omega$ and $r\in [0, r_\omega]$,
	$$\frac{p(\mx(r, \omega))}{q(\mx(r, \omega))}\log\frac{p(\mx(r, \omega))}{q(\mx(r, \omega))} - \frac{p(\mx(r, \omega))}{q(\mx(r, \omega))} + 1\le 9M^2\left(\sqrt{\frac{p(\mx(r, \omega))}{q(\mx(r, \omega))}} - 1\right)^2,$$
	which indicates that
	\begin{align*}
		&\quad \int_\Omega\int_0^{r_\omega} r^{d-1}\left(p(\mx(r, \omega))\log\frac{p(\mx(r, \omega))}{q(\mx(r, \omega))} - p(\mx(r, \omega)) + q(\mx(r, \omega))\right)drd\omega\\
		& = \int_{\omega\in \Omega}\int_0^{r_\omega} r^{d-1}q(\mx(r, \omega))\cdot\left(\frac{p(\mx(r, \omega))}{q(\mx(r, \omega))}\log\frac{p(\mx(r, \omega))}{q(\mx(r, \omega))} - \frac{p(\mx(r, \omega))}{q(\mx(r, \omega))} + 1\right)drd\omega\\
		&\le 9M^2\int_\Omega\int_0^{r_\omega} r^{d-1}q(\mx(r, \omega))\cdot\left(\sqrt{\frac{p(\mx(r, \omega))}{q(\mx(r, \omega))}} - 1\right)^2drd\omega\\
		& \le 9M^2\int_\Omega\int_0^{\infty} r^{d-1}q(\mx(r, \omega))\cdot\left(\sqrt{\frac{p(\mx(r, \omega))}{q(\mx(r, \omega))}} - 1\right)^2drd\omega = 9M^2H^2(p, q).
	\end{align*}
	\par Combine the above two cases, we obtain that
	\begin{align*}
		\KL(p\|q) & = \int_\Omega\int_0^\infty r^{d-1}\left(p(\mx(r, \omega))\log\frac{p(\mx(r, \omega))}{q(\mx(r, \omega))} - p(\mx(r, \omega)) + q(\mx(r, \omega))\right)drd\omega\\
		& = \int_\Omega\int_0^{r_\omega} r^{d-1}\left(p(\mx(r, \omega))\log\frac{p(\mx(r, \omega))}{q(\mx(r, \omega))} - p(\mx(r, \omega)) + q(\mx(r, \omega))\right)drd\omega\\
		&\quad + \int_\Omega\int_{r_\omega}^\infty r^{d-1}\left(p(\mx(r, \omega))\log\frac{p(\mx(r, \omega))}{q(\mx(r, \omega))} - p(\mx(r, \omega)) + q(\mx(r, \omega))\right)drd\omega\\
		& \le 18M^2 H^2(p, q) + 5136M^2 H^2(p, q) = 5154M^2H^2(p, q).
	\end{align*}
	This completes the proof of Theorem \ref{thm-KL-d}.
\end{proof}

\section{Proof of Theorem \ref{thm-sub-lb}} \label{app-3}
	\par Given some constant $r > 1$, we  consider the following distribution:
	$$\pi_r = (1-h_r)\delta_0 + h_r\delta_r,$$
	where $h_r = \exp\left(-\frac{r^2}{2K^2}\right).$ Then for any $r > 1$, $\pi_r$ is a $K$-subgaussian distribution.
	\par We let $p_r = \pi_r * \mathcal{N}(0, 1)$. Since $x\log x - x + 1\ge 0$ holds for all $x > 0$, we have
	\begin{align*}
		& \quad \KL(p_r\|\mathcal{N}(0, 1))\\
		& = \int_{-\infty}^\infty p_r(x)\log\frac{p_r(x)}{\varphi(x)}dx = \int_{-\infty}^\infty \varphi(x)\cdot \left(\frac{p_r(x)}{\varphi(x)}\log\frac{p_r(x)}{\varphi(x)} - \frac{p_r(x)}{\varphi(x)} + 1\right)dx\\
		& \ge \int_{r}^{r+1} \varphi(x)\cdot \left(\frac{p_r(x)}{\varphi(x)}\log\frac{p_r(x)}{\varphi(x)} - \frac{p_r(x)}{\varphi(x)} + 1\right)dx\ge \int_r^{r+1}p_r(x)\log\frac{p_r(x)}{\varphi(x)} - p_r(x)dx.
	\end{align*}
	According to our construction, for $r\le x\le r+1$ we have
	$$p_r(x) = (1-h_r)\varphi(x) + h_r\varphi(r-x)\ge h_r\cdot \varphi(1) \quad \text{ and also }\quad \varphi(x)\le \varphi(r).$$
	Therefore, we obtain that for $r\le x\le r+1$,
	$$\log\frac{p_r(x)}{\varphi(x)}\le \log \frac{p_r\varphi(1)}{\varphi(r)} = \log\exp\left(-\frac{r^2}{2K^2} - \frac{1}{2} + \frac{r^2}{2}\right) = \frac{r^2}{2} - \frac{r^2}{2K^2} - \frac{1}{2}.$$
	Noticing that $\varphi(1) = \frac{1}{\sqrt{2\pi}}\exp\left(-\frac{1}{2}\right)\ge \frac{1}{5}$, we obtain that
\begin{equation}
\KL(p_r\|\mathcal{N}(0, 1))\ge \int_r^{r+1}\frac{h_r}{5}\cdot \left(\frac{r^2}{2} - \frac{r^2}{2K^2} - \frac{1}{2} - 1\right)dx = \left(\frac{r^2}{10} - \frac{r^2}{10K^2} - \frac{3}{10}\right)h_r.
    \label{eq:KLlower}
\end{equation}	

Next, we write
	$$H^2(p_r, \mathcal{N}(0, 1)) = \int_{-\infty}^\infty \left(\sqrt{p_r(x)} - \sqrt{\varphi(x)}\right)^2dx.$$
	We divide the integral domain into three regions: $(-\infty, -r], [-r, r]$ and $[r, \infty)$, and upper bound the contribution  from each region separately. 
 
 Noticing that $p_r(x) = (1-h_r)\varphi(x) + h_r\varphi(r-x)$, we have for any $x\le -r$, 
	$$0\le p_r(x)\le \varphi(x).$$
	Therefore,
	\begin{align*}
		&\quad \int_{-\infty}^{-r} \left(\sqrt{p_r(x)} - \sqrt{\varphi(x)}\right)^2dx \le \int_{-\infty}^{-r}\varphi(x)dx = \frac{1}{\sqrt{2\pi}}\int_{-\infty}^{0}\exp\left(-\frac{(x+r)^2}{2}\right)dx\\
		& \le \exp\left(-\frac{r^2}{2}\right)\int_{-\infty}^0\exp(-rx)dx = \frac{1}{r}\exp\left(-\frac{r^2}{2}\right)\le \exp\left(-\frac{r^2}{2K^2}\right) = h_r.
	\end{align*}
 
	Noticing that for those $x\ge r$,
	$$(1 - h_r)\varphi(x)\le \varphi(x) = \frac{1}{\sqrt{2\pi}}\exp\left(-\frac{x^2}{2}\right)\le \frac{1}{\sqrt{2\pi}}\exp\left(-\frac{r^2 + (x - r)^2}{2}\right)\le h_r\varphi(r - x),$$
	we obtain 
	$$0\le \varphi(x)\le (1-h_r)\varphi(x) + h_r\varphi(r-x) = p_r(x)\le 2h_r\varphi(r-x) \qquad \forall x\ge r,$$
	which indicates that
	$$\int_{r}^{\infty} \left(\sqrt{p_r(x)} - \sqrt{\varphi(x)}\right)^2dx = \int_{r}^\infty p_r(x)dx\le 2h_r\int_{r}^\infty \varphi(r-x)dx = h_r.$$

 Finally for those $-r\le x\le r$, we notice that 
	$$p_r(x) - \varphi(x) = (1-h_r)\varphi(x) + h_r\varphi(r-x) - \varphi(x) = h_r(\varphi(r-x) - \varphi(x)),$$
	hence $|p_r(x) - \varphi(x)|\le h_r\cdot |\varphi(r-x) - \varphi(x)|\le \frac{h_r}{\sqrt{2\pi}}\le h_r$. Therefore, if either $p_r(x)\ge h_r$ or $\varphi(x)\ge h_r$, we have
	$$\left(\sqrt{p_r(x)} - \sqrt{\varphi(x)}\right)^2 = \frac{(p_r(x) - \varphi(x))^2}{\left(\sqrt{p_r(x)} + \sqrt{\varphi(x)}\right)^2}\le \frac{h_r^2}{\sqrt{h_r}^2} = h_r.$$
	And if neither $p_r(x)\ge h_r$ nor $\varphi(x)\ge h_r$ holds, then we have $0\le p_r(x), \varphi(x)\le h_r$ and hence
	$$\left(\sqrt{p_r(x)} - \sqrt{\varphi(x)}\right)^2\le h_r.$$
	Overall, we have $\left(\sqrt{p_r(x)} - \sqrt{\varphi(x)}\right)^2\le h_r$ and hence
	$$\int_{-r}^r \left(\sqrt{p_r(x)} - \sqrt{\varphi(x)}\right)^2dx\le 2rh_r.$$
 
	Combining the contributions from these regions, we obtain that
\begin{equation}
    \label{eq:Hupper}
     H^2(p_r, \mathcal{N}(0, 1)) = \int_{-\infty}^\infty\left(\sqrt{p_r(x)} - \sqrt{\varphi(x)}\right)^2dx\le h_r + 2rh_r + h_r = (2 + 2r)h_r.
\end{equation}

 Finally, applying \prettyref{eq:KLlower} and \prettyref{eq:Hupper}, we choose the parameter $r$ so that $\KL/H^2$ exceeds an arbitrary constant $C$. Noticing that $K > 1$, we have $\frac{r^2}{10} - \frac{r^2}{10K^2}\ge 0$, hence there exists $r > 1$ such that
	$$\frac{r^2}{10} - \frac{r^2}{10K^2} - \frac{3}{10}\ge c\cdot (2 + 2r).$$
	And for this $r$, we have
	$$\KL(p_r\|\mathcal{N}(0, 1))\ge c\cdot H^2(p_r, \mathcal{N}(0, 1)).$$
	This finishes the proof of Theorem \ref{thm-sub-lb}.

\begin{remark}
    Notice that with similar proving techniques, we can also show that 
\end{remark}

\section{Proof of Theorem \ref{thm-sub-ub}} \label{app-4}
\par Without loss of generality, we assume $d\le \frac{1}{2(1-K)}$ (otherwise we use $1 - \frac{1}{2d}\ge K$ to replace $K$, and since $\pi, \eta$ are $K$-subgaussian, hence they are also $\left(1 - \frac{1}{2d}\right)$-subgaussian. The results still hold). We abbreviate $f_\pi(\cdot), f_\eta(\cdot)$ as $p(\cdot), q(\cdot)$.
\begin{lemma}\label{lem-sub-eq1}
	Suppose $p = \pi * \mathcal{N}(0, I_d)$, where $\pi$ is a $K$-subgaussian distribution with $K < 1$. Then for every $r\ge \frac{2\sqrt{K}}{1 - \sqrt{K}}$ and any $\omega\in \Omega$, we have the following propositions:
	\begin{enumerate}
		\item $\forall r'\in \left[\frac{2\sqrt{K}}{1 - \sqrt{K}}, r\right]$, we have 
		\begin{equation}\label{eq-sub-eq1}p(\mx(r', \omega))\ge \frac{p(\mx(r, \omega))}{7}.\end{equation}
		\item $\forall r'\ge r$, we have 
		\begin{equation}\label{eq-sub-eq2}p(\mx(r', \omega))\le 7p(\mx(r, \omega))\cdot \exp\left(-\frac{1 - K}{4}r(r' - r)\right).\end{equation}
	\end{enumerate}
\end{lemma}
\begin{proof}
For every $r'\ge r$ and $\omega\in \Omega$, notice that we can write $p(\cdot)$ as the following integral:
	$$p(\mx(r', \omega)) = \int_{\RR^d} \varphi(\mx(r', \omega) - \mmu)\pi(\mmu)d\mmu,\qquad p(\mx(r, \omega)) = \int_{\RR^d} \varphi(\mx(r, \omega) - \mmu)\pi(\mmu)d\mmu,$$
	where we use $\pi(\cdot)$ to denote the density distribution of $\pi$ (which can be a generalized function), and $\varphi(\cdot)$ to denote the density distribution of $\mathcal{N}(0, I_d)$.
	
	Since $\pi$ is a $K$-subgaussian distribution, we have $\mathbf{P}[\|X\|_2\ge \alpha]\le \exp\left(-\frac{\alpha^2}{2K^2}\right)$ for $\alpha\ge 0$ and $\mathbf{P}[\|X\|_2\ge 1]\le \frac{1}{\sqrt{e}}\le \frac{2}{3}$, where $X\sim \pi$. Therefore, $\mathbf{P}[\|X\|_2\le 1]\ge \frac{1}{3}$. Given $0\le \alpha\le r\le \beta\le r'$ ($\alpha, \beta$ will be specified later),  we have
	\begin{align*}
	\int_{\alpha\le \|\mmu\|\le \beta}\varphi(\mx(r', \omega) - \mmu)\pi(\mmu)d\mmu & \le \frac{1}{\sqrt{2\pi}^d}\exp\left(-\frac{(r' - \beta)^2}{2}\right)\cdot \mathbf{P}[\|X\|_2\ge \alpha]\\
	&\le \frac{1}{\sqrt{2\pi}^d}\exp\left(-\frac{\alpha^2}{2K^2} - \frac{(r' - \beta)^2}{2}\right);\\
	p(\mx(r, \omega)) & \ge \frac{1}{\sqrt{2\pi}^d}\exp\left(-\frac{(r + 1)^2}{2}\right)\cdot \mathbf{P}[\|X\|_2\le 1]\\
	&\ge \frac{1}{3}\cdot \frac{1}{\sqrt{2\pi}^d}\exp\left(-\frac{(r+1)^2}{2}\right),
	\end{align*}
	which indicates that
	$$\int_{\alpha\le \|\mmu\|\le \beta}\varphi(\mx(r', \omega) - \mmu)\pi(\mmu)d\mmu\le p(\mx(r, \omega))\cdot 3\exp\left(\frac{(r+1)^2}{2} - \frac{\alpha^2}{2K^2} - \frac{(r' - \beta)^2}{2}\right).$$
	We further have
	$$\int_{\|\mmu\|_2\ge \beta}\varphi(\mx(r', \omega) - \mmu)\pi(\mmu)d\mmu\le \frac{1}{\sqrt{2\pi}^d}\cdot \mathbf{P}[\|X\|_2\ge \beta]\le \frac{1}{\sqrt{2\pi}^d}\exp\left(-\frac{\beta^2}{2K^2}\right),$$
	which indicates that
	$$\int_{\alpha\le \|\mmu\|\le \beta}\varphi(\mx(r', \omega) - \mmu)\pi(\mmu)d\mmu\le p(\mx(r, \omega))\cdot 3\exp\left(\frac{(r+1)^2}{2} - \frac{\beta^2}{2K^2}\right)$$
	Finally, for every $\mx$ such that $\|\mx\|_2\le \alpha$, similar to previous proof we have 
	$$\|\mx - \mx(r, \omega)\|_2^2 - (r - \alpha)^2\le \|\mx - \mx(r', \omega)\|_2^2 - (r' - \alpha)^2,$$
	which indicates that
	\begin{align*}
		\int_{\|u\|\le \alpha}\varphi(\mx(r', \omega) - \mmu)\pi(\mmu)d\mmu & \le \int_{\|u\|\le \alpha}\varphi(\mx(r, \omega) - \mmu)\pi(\mmu)d\mmu\cdot \exp\left(\frac{(r-\alpha)^2}{2} - \frac{(r'-\alpha)^2}{2}\right)\\
		& \le p(\mx(r, \omega))\exp\left(\frac{(r-\alpha)^2}{2} - \frac{(r'-\alpha)^2}{2}\right).
	\end{align*}
	Above all, we get
	\begin{align*}
		p(\mx(r', \omega)) & \le p(\mx(r, \omega))\cdot \Bigg(3\exp\left(\frac{(r+1)^2}{2} - \frac{\alpha^2}{2K^2} - \frac{(r' - \beta)^2}{2}\right)\\
		&\quad + 3\exp\left(\frac{(r+1)^2}{2} - \frac{\beta^2}{2K^2}\right) + \exp\left(\frac{(r-\alpha)^2}{2} - \frac{(r'-\alpha)^2}{2}\right)\Bigg)
	\end{align*}
	\par Further noticing that when $r\ge \frac{2\sqrt{K}}{1 - \sqrt{K}}$, we have $$r - \sqrt{K}(1 + r)\ge \frac{1 - \sqrt{K}}{2}r\ge 0.$$
	Therefore, choosing $\alpha = \sqrt{K}(r+1)$ and $\beta = \frac{r' + r}{2}$, and noticing that $2(1-\sqrt{K})\ge 1-K$ holds for all $0 < K < 1$, we will get
	\begin{align*}
	\frac{(r-\alpha)^2}{2} - \frac{(r' - \alpha)^2}{2} & = \exp\left(-\frac{(r' - r)^2}{2} - (r' - r)(r - \sqrt{K}(r+1))\right)\\
	\le \exp\left(-\frac{1 - \sqrt{K}}{2}r(r' - r)\right) & \le \exp\left(-\frac{1 - K}{4}r(r' - r)\right),\\
	\frac{(r+1)^2}{2} - \frac{\alpha^2}{2K^2} - \frac{(r' - \beta)^2}{2} & = -\frac{(r+1)^2(1-K)}{2K} - \frac{(r' - r)^2}{8}\le -\frac{(1-K)r^2}{2} - \frac{(r' - r)^2}{8},\\
	\frac{(r+1)^2}{2} - \frac{\beta^2}{2K^2} & \le \frac{(r+1)^2}{2} - \frac{r^2}{2K^2} - \frac{(r' - r)^2}{8K^2}\le \frac{(r+1)^2}{2} - \frac{(r+1)^2}{2K} - \frac{(r' - r)^2}{8K^2}\\
	& \le -\frac{(1-K)r^2}{2} - \frac{(r' - r)^2}{8}.
	\end{align*}
	Hence we get
	$$p(\mx(r', \omega))\le p(\mx(r, \omega))\cdot \left(6\exp\left(-\frac{(1-K)r^2}{2} - \frac{(r' - r)^2}{8}\right) + \exp\left(-\frac{1 - K}{4}r(r' - r)\right)\right).$$
	Next we notice that
	$$\frac{(1-K)r^2}{2} + \frac{(r' - r)^2}{8}\ge (1-K)\cdot \left(\frac{r^2}{2} + \frac{(r' - r)^2}{8}\right)\ge 2(1-K)\frac{r(r'-r)}{4}\ge \frac{(1-K)r(r'-r)}{4},$$
	which indicates that
	$$p(\mx(r', \omega))\le 7p(\mx(r, \omega))\cdot\exp\left(-\frac{1 - K}{4}r(r' - r)\right).$$
	This proves \eqref{eq-sub-eq2}.
	
 Finally, swapping $r$ and $r'$ in \eqref{eq-sub-eq2}, and noticing that 
	$$0\le \exp\left(-\frac{1 - K}{4}r(r' - r)\right)\le 1,$$
	we get \eqref{eq-sub-eq1}.
\end{proof}
\begin{lemma}\label{lem-sub-eq2}
	Suppose $p = \pi * \mathcal{N}(0, I_d)$, where $\pi$ is a $K$-subgaussian distribution with $K < 1$. For $\mx\in\mathbb{R}^d$, we have
	$$\|\nabla \log p(\mx)\|_2\le 3\|\mx\|_2 + 12, \quad \|\nabla \log p(\mx)\|_2\le 3\|r\|_2 + 12.$$
\end{lemma}
\begin{proof}
	According to Proposition 2 in~\cite{polyanskiy2016wasserstein}, we only need to verify $\mathbb{E}[\|X\|_2]\le 3$,	where $X\sim \pi$. Indeed, applying $K$-subgaussianity, we have	
	$$\mathbb{E}[\|X\|_2] = \int_0^\infty\mathbf{P}[\|X\|_2\ge r]dr\le \int_0^\infty\exp\left(-\frac{r^2}{2K^2}\right)dr = \sqrt{2\pi}K\le 3.$$
\end{proof}

\begin{proof}[Proof of Theorem \ref{thm-sub-ub}]
	First we can write
	\begin{align}\label{eq-sub-polar}
		\KL(p\|q) & = \int_\Omega\int_0^\infty r^{d-1}p(\mx(r, \omega))\log\frac{p(\mx(r, \omega))}{q(\mx(r, \omega))}drd\omega\\
		& = \int_\Omega\int_0^\infty r^{d-1}\left(p(\mx(r, \omega))\log\frac{p(\mx(r, \omega))}{q(\mx(r, \omega))} - p(\mx(r, \omega)) + q(\mx(r, \omega))\right)drd\omega\\
		H^2(p, q) & = \int_\Omega\int_0^\infty r^{d-1}p(\mx(r, \omega))\left(\sqrt{\frac{q(\mx(r, \omega))}{p(\mx(r, \omega))}} - 1\right)^2drd\omega
	\end{align}
	For every $\omega\in\Omega$, we define $r_\omega$ as
	$$r_\omega = \inf\left\{r\colon\log \frac{p(\mx(r_\omega, \omega))}{q(\mx(r_\omega, \omega))}\ge \log 3 + \frac{(3 - \sqrt{K})^2}{2(1-\sqrt{K})^2}\right\}.$$
	We notice that for every $\mx\in \RR^d$, 
	$$p(\mx) = \int_{\RR^d} \pi(\my)\varphi(\mx - \my)d\my\le \frac{1}{\sqrt{2\pi}^d}\int_{\RR^d} \pi(\my)d\my = \frac{1}{\sqrt{2\pi}^d},$$
	and we further have 
	\begin{align*}
		q(\mx) & = \int_{\RR^d}\eta(\my)\varphi(\mx - \my)d\my\ge \int_{\|\my\|_2\le 1}\eta(\my)\varphi(\mx - \my)d\my\\
		&\ge (1 - e^{-1/2})\cdot \min_{\|\my\|_\le 1} \varphi(\mx - \my)\ge \frac{1}{3}\cdot \frac{1}{\sqrt{2\pi}^d}\exp\left(-\frac{(\|\mx\| + 1)^2}{2}\right).
	\end{align*}
	Therefore, for all $\mx$ such that $\|\mx\|_2 < \frac{2}{1 - \sqrt{K}}$, we have
	\begin{equation}\label{eq-def-T}\log \frac{p(\mx(r_\omega, \omega))}{q(\mx(r_\omega, \omega))} < \log 3 + \frac{(3 - \sqrt{K})^2}{2(1-\sqrt{K})^2}\le \log 3 + \frac{18}{(1-K)^2}\le \frac{20}{(1-K)^2}.\end{equation}
	Hence for every $\omega\in \Omega$, we have
	\begin{equation}\label{eq-sub-r}r_\omega \ge \frac{2}{1 - \sqrt{K}} = \frac{2(1 + \sqrt{K})}{1 - K}\ge \frac{2}{1 - K}.\end{equation}

	\par Then according to Lemma \ref{lem-sub-eq1} and Lemma \ref{lem-sub-eq2}, we know that for every $r\ge r_\omega$, we have
	\begin{align*}
		p(\mx(r, \omega)) & \le 7p(\mx(r_\omega, \omega))\exp\left(-\frac{1 - K}{4}r_\omega(r-r_\omega)\right),\\
		\log \frac{p(\mx(r, \omega))}{q(\mx(r, \omega))} & \le T + (r - r_\omega)(3r + 3r_\omega + 24),
	\end{align*}
	where $T$ is defined in \eqref{eq-def-T}. Therefore, we obtain that
	\begin{equation}\label{eq-sub-pr}\begin{aligned}
		&\quad \int_{r_\omega}^\infty r^{d-1}p(\mx(r, \omega))\log \frac{p(\mx(r, \omega))}{q(\mx(r, \omega))}dr\\
		&\le 7p(\mx(r_\omega, \omega))\int_{r_\omega}^\infty r^{d-1}\left(\frac{20}{(1 - K)^2} + (r-r_\omega)(3r + 3r_\omega + 24)\right)\cdot\exp\left(-\frac{1 - K}{4}r_\omega(r-r_\omega)\right)dr.
	\end{aligned}\end{equation}
	We adopt the changes of variables from $r$ to $t = r - r_\omega$, and obtain that
	$$\int_{r_\omega}^\infty r^{d-1}\exp\left(-\frac{1 - K}{4}r_\omega(r-r_\omega)\right)dr =  r_\omega^{d-1}\int_0^\infty\exp\left((d-1)\log\frac{t + r_\omega}{r_\omega} - \frac{(1-K)tr_\omega}{4}\right)dt$$
	and
	\begin{align*}
		&\quad \int_{r_\omega}^\infty r^{d-1}(r-r_\omega)(3r + 3r_\omega + 24)\exp\left(-\frac{1 - K}{4}r_\omega(r-r_\omega)\right)dr\\
		& = r_\omega^{d-1}\int_0^\infty\exp\left((d-1)\log\frac{t + r_\omega}{r_\omega} + \log t + \log (3t + 6r_\omega + 24) - \frac{(1-K)tr_\omega}{4}\right)dt
	\end{align*}
	\par We first bound the first term in \eqref{eq-sub-pr}. We define
	$$f(t) \triangleq (d-1)\log\frac{t + r_\omega}{r_\omega} - \frac{1-K}{8}r_\omega t,$$
	then noticing that $d\le \frac{1}{2(1-K)} + 1$ and hence $8(d - 1)\le \frac{4}{1-K}\le (1-K)r_\omega^2$ holds for $r_\omega\ge \frac{2}{1 - \sqrt{K}}\ge \frac{2}{1-K}$, its derivative satisfies that 
	$$f'(t) = \frac{d-1}{t + r_\omega} - \frac{(1-K)r_\omega}{8}\le \frac{d-1}{r_\omega} - \frac{(1-K)r_\omega}{8}\le 0, \quad \forall t\ge 0.$$
	Therefore, for every $t\ge 0$, we have
	$$(d-1)\log\frac{t + r_\omega}{r_\omega} - \frac{1-K}{8}r_\omega t = f(t)\le f(0) = 0,$$
	which indicates that
	\begin{equation}\label{eq-sub-q}\quad \frac{1}{r_\omega^{d-1}}\int_{r_\omega}^\infty r^{d-1}\exp\left(-\frac{1 - K}{4}r_\omega(r-r_\omega)\right)dr \le \int_0^\infty\exp\left(- \frac{1-K}{8}r_\omega t\right)dt = \frac{8}{(1-K)r_\omega}\end{equation}
	\par We next bound the second term in \eqref{eq-sub-pr}. We define
	$$g(t)\triangleq \log t + \log (3t + 6r_\omega + 24) - \frac{1-K}{16}r_\omega t,$$
	then we have
	$$g'(t) = \frac{1}{t} + \frac{1}{3t + 6r_\omega + 24} - \frac{1-K}{16}r_\omega,$$
	which has a single root $t_0$ on $(0, \infty)$. And we have $g(t)\le g(t_0)$ holds for all $t\ge 0$. We further have 
	$$0 = g'(t_0) = \frac{1}{t_0} + \frac{1}{3t_0 + 6r_\omega + 24} - \frac{1-K}{16}r_\omega\le \frac{4}{3t_0} - \frac{1-K}{16}r_\omega,$$
	which indicates that $t_0\le \frac{64}{3(1-K)r_\omega}$. Next noticing that $r_\omega\ge \frac{2}{1 - \sqrt{K}}$, we have $(1-K)r_\omega \ge 2(1 + \sqrt{K})\ge 2$, hence we get
	$$3t_0 + 6r_\omega + 24\le \frac{64}{(1-K)r_\omega} + 6r_\omega + 24\le 6r_\omega + 56\le 34r_\omega.$$
	Therefore for all $t\ge 0$,
	$$g(t)\le g(t_0)\le \log t_0(3t_0 + 6r_\omega + 24) - \frac{4}{3}\le \log \frac{2176}{3(1-K)} - \frac{4}{3},$$
	Combine this result with our previous estimation on $f$, we get
	\begin{align*}
		&\quad (d-1)\log\frac{t + r_\omega}{r_\omega} + \log t + \log (3t + 6r_\omega + 24) - \frac{1-K}{4}r_\omega t\\
		& = f(t) + g(t) - \frac{1-K}{16}r_\omega t\le \log \frac{2176}{3(1-K)} - \frac{4}{3} - \frac{(1-K)r_\omega t}{16}.
	\end{align*}
	Hence we obtain
	\begin{align*}
		&\quad \int_{r_\omega}^\infty r^{d-1}(r-r_\omega)(3r + 3r_\omega + 24)\exp\left(-\frac{1 - K}{4}r_\omega(r-r_\omega)\right)dr\\
		&\le r_\omega^{d-1}\int_0^\infty \exp\left(\log \frac{2176}{3(1-K)} - \frac{4}{3} - \frac{(1-K)r_\omega t}{16}\right)dt = \frac{2176e^{-4/3}}{3(1-K)}r_\omega^{d-1}\cdot \frac{16}{(1-K)r_\omega}\le \frac{3100r_\omega^{d-2}}{(1-K)^2}.
	\end{align*}
	Therefore, according to \eqref{eq-sub-pr} we have
	\begin{align*}
		&\quad \int_{r_\omega}^\infty r^{d-1}p(\mx(r, \omega))\log \frac{p(\mx(r, \omega))}{q(\mx(r, \omega))}dr\le  \left(\frac{160r_\omega^{d-2}}{(1-K)^3} + \frac{3100r_\omega^{d-2}}{(1-K)^2}\right)\cdot 7p(\mx(r_\omega, \omega))\le \frac{23000r_\omega^{d-2}p(\mx(r_\omega, \omega))}{(1-K)^3},
	\end{align*}
	\par Similarly, we can also obtain bound on $\int_{r_\omega}^\infty r^{d-1}q(\mx(r, \omega))dr$: According to Lemma \ref{lem-sub-eq1} we obtain that
	$$q(\mx(r, \omega)) \le 7q(\mx(r_\omega, \omega))\exp\left(-\frac{1 - K}{4}r_\omega(r - r_\omega)\right),$$
	which indicates that
	$$\int_{r_\omega}^\infty r^{d-1}q(\mx(r, \omega))dr\le q(\mx(r_\omega, \omega))\int_{r_\omega}^\infty r^{d-1}\exp\left(-\frac{1 - K}{4}r_\omega(r - r_\omega)\right)dr.$$
	According to \eqref{eq-sub-q}, we get
	$$\int_{r_\omega}^\infty r^{d-1}\exp\left(-\frac{1 - K}{4}r_\omega(r - r_\omega)\right)dr\le \frac{8r_\omega^{d-2}}{1 - K},$$
	which indicates that
	$$\int_{r_\omega}^\infty r^{d-1}q(\mx(r, \omega))dr\le \frac{56r_\omega^{d-2}q(\mx(r_\omega, \omega))}{1-K}$$
	Next noticing $\log \frac{p(\mx(r_\omega, \omega))}{q(\mx(r_\omega, \omega))} = \log 3 + \frac{(3 - \sqrt{K})^2}{2(1-\sqrt{K})^2}\ge 0$. we have $q(\mx(r_\omega, \omega))\le p(\mx(r_\omega, \omega))$. Therefore, we have
	$$\int_{r_\omega}^\infty r^{d-1}q(\mx(r, \omega))dr\le \frac{56r_\omega^{d-2}p(\mx(r_\omega, \omega))}{1-K}\le \frac{56r_\omega^{d-2}p(\mx(r_\omega, \omega))}{(1-K)^3}.$$
	Therefore, we have the following upper bound 
	\begin{align*}
	&\quad \int_\Omega\int_{r_\omega}^\infty r^{d-1}\left(p(\mx(r, \omega))\log\frac{p(\mx(r, \omega))}{q(\mx(r, \omega))} - p(\mx(r, \omega)) + q(\mx(r, \omega))\right)drd\omega\\
	& \le \int_\Omega\int_{r_\omega}^\infty r^{d-1}p(\mx(r, \omega))\log\frac{p(\mx(r, \omega))}{q(\mx(r, \omega))}drd\omega + \int_\Omega\int_{r_\omega}^\infty r^{d-1}q(\mx(r, \omega))drd\omega\\
	& \le \int_\Omega\frac{23056r_\omega^{d-2}p(\mx(r_\omega, \omega))}{(1-K)^3}d\omega.
	\end{align*}

	\par Next, according to \eqref{lem-sub-eq2}, we notice that for any $\omega\in \Omega$ and $0\le r\le r_\omega$ we have 
	$$\nabla \log\frac{p(\mx(r, \omega))}{q(\mx(r, \omega))}\le 6r + 24\le 6r_\omega + 24.$$
	According to our choice of $r_\omega$, we have $\log\frac{p(\mx(r_\omega, \omega))}{q(\mx(r_\omega, \omega))} = \log 3 + \frac{(3 - \sqrt{K})^2}{2(1-\sqrt{K})^2}$. Hence noticing that $r_\omega\ge 2$ for every $\omega\in \Omega$ according to \eqref{eq-sub-r}, we have
	for any $r_\omega - \frac{1}{18r_\omega}\le r\le r_\omega$, 
	$$\log\frac{p(\mx(r, \omega))}{q(\mx(r, \omega))}\ge \log 3 + \frac{(3 - \sqrt{K})^2}{2(1-\sqrt{K})^2} - \frac{6r_\omega + 24}{18r_\omega}\ge 1 + 2 - 1 = 2,$$
	which indicates that
	$$\left(\sqrt{\frac{q(\mx(r, \omega))}{p(\mx(r, \omega))}} - 1\right)^2\ge \left(\frac{1}{e^1} - 1\right)^2\ge \frac{1}{3}.$$
	Additionally, according to Lemma \ref{lem-eq}, for every $r_\omega - \frac{1}{18r_\omega}\le r\le r_\omega$, we have $p(\mx(r, \omega))\ge \frac{p(\mx(r_\omega, \omega))}{7}$. We further adopt the assumption $d\le \frac{1}{2(1-K)} + 1$ and also use \eqref{eq-sub-r} to get
	$$r^{d-1}\ge r_\omega^{d-1}\cdot \left(1 - \frac{1}{18r_\omega^2}\right)^{d-1}\ge r_\omega^{d-1}\cdot \left(1 - \frac{d-1}{18r_\omega^2}\right)\ge r_\omega^{d-1}\left(1 - \frac{1}{36}\right)\ge \frac{1}{2}r_\omega^{d-1}.$$
	Therefore, we have
	\begin{align*}
		&\quad \int_\Omega\int_{r_\omega - \frac{1}{18r_\omega}}^{r_\omega}r^{d-1}p(\mx(r, \omega))\cdot \left(\sqrt{\frac{q(\mx(r, \omega))}{p(\mx(r, \omega))}} - 1\right)^2drd\omega\\
		& \ge \int_\Omega\frac{1}{18r_\omega}\cdot \frac{1}{2}r_\omega^{d-1}\cdot p(\mx(r_\omega, \omega))\cdot \frac{1}{2}\ge \frac{r_\omega^{d-2}p(\mx(r_\omega, \omega))}{72}d\omega\\
		& \ge \frac{(1-K)^3}{1660032}\int_\Omega\int_{r_\omega}^\infty r^{d-1}\left(p(\mx(r, \omega))\log\frac{p(\mx(r, \omega))}{q(\mx(r, \omega))} - p(\mx(r, \omega)) + q(\mx(r, \omega))\right)drd\omega.
	\end{align*}
	Therefore, according to \eqref{eq-sub-polar}, we have
	\begin{align*}
		H^2(p, q) & = \int_\Omega\int_0^\infty r^{d-1}p(\mx(r, \omega))\left(\sqrt{\frac{q(\mx(r, \omega))}{p(\mx(r, \omega))}} - 1\right)^2drd\omega\\
		& \ge \int_\Omega\int_{r_\omega - \frac{1}{r_\omega + M}}^{r_\omega} r^{d-1}p(\mx(r, \omega))\left(\sqrt{\frac{q(\mx(r, \omega))}{p(\mx(r, \omega))}} - 1\right)^2drd\omega\\
		& \ge \frac{(1-K)^3}{1660032}\int_\Omega\int_{r_\omega}^\infty r^{d-1}\left(p(\mx(r, \omega))\log\frac{p(\mx(r, \omega))}{q(\mx(r, \omega))} - p(\mx(r, \omega)) + q(\mx(r, \omega))\right)drd\omega.
	\end{align*}

	\par Next we consider those $\mx(r, \omega)$ with $0\le r\le r_\omega$. According to the definition of $r_\omega$, we know that for any such $r$, we have 
	$$\log\frac{p(\mx(r, \omega))}{q(\mx(r, \omega))}\le  \log 3 + \frac{(3 - \sqrt{K})^2}{2(1-\sqrt{K})^2}\le \frac{20}{(1-K)^2}.$$
	According to Lemma \ref{lem-formula}, we have
	\begin{align*}
		\frac{p(\mx(r, \omega))}{q(\mx(r, \omega))}\log\frac{p(\mx(r, \omega))}{q(\mx(r, \omega))} - \frac{p(\mx(r, \omega))}{q(\mx(r, \omega))} + 1 & \le \frac{22.5}{(1-K)^2}\left(\sqrt{\frac{p(\mx(r, \omega))}{q(\mx(r, \omega))}} - 1\right)^2\\
		& \le \frac{24}{(1-K)^2}\left(\sqrt{\frac{p(\mx(r, \omega))}{q(\mx(r, \omega))}} - 1\right)^2.
	\end{align*}
	Hence we obtain that
	\begin{align*}
		&\quad\int_0^{r_\omega} r^{d-1}\left(p(\mx(r, \omega))\log\frac{p(\mx(r, \omega))}{q(\mx(r, \omega))} - p(\mx(r, \omega)) + q(\mx(r, \omega))\right)dr\\
		& = \int_0^{r_\omega} r^{d-1}q(\mx(r, \omega))\cdot\left(\frac{p(\mx(r, \omega))}{q(\mx(r, \omega))}\log\frac{p(\mx(r, \omega))}{q(\mx(r, \omega))} - \frac{p(\mx(r, \omega))}{q(\mx(r, \omega))} + 1\right)dr\\
		& \le \int_0^{r_\omega} r^{d-1}q(\mx(r, \omega))\cdot \frac{24}{(1-K)^2}\left(\sqrt{\frac{p(\mx(r, \omega))}{q(\mx(r, \omega))}} - 1\right)^2dr.
	\end{align*}
	Therefore, we have
	\begin{align*}
		&\quad \int_\Omega\int_0^{r_\omega} r^{d-1}\left(p(\mx(r, \omega))\log\frac{p(\mx(r, \omega))}{q(\mx(r, \omega))} - p(\mx(r, \omega)) + q(\mx(r, \omega))\right)drd\omega\\
		&\le \frac{24}{(1-K)^2}\int_\Omega\int_0^{r_\omega} r^{d-1}q(\mx(r, \omega))\cdot\left(\sqrt{\frac{p(\mx(r, \omega))}{q(\mx(r, \omega))}} - 1\right)^2drd\omega\\
		& \le \frac{24}{(1-K)^2}\int_\Omega\int_0^{\infty} r^{d-1}q(\mx(r, \omega))\cdot\left(\sqrt{\frac{p(\mx(r, \omega))}{q(\mx(r, \omega))}} - 1\right)^2drd\omega = \frac{24}{(1-K)^2}M^2H^2(p, q).
	\end{align*}
	\par Combine these two analysis together, we obtain that
	\begin{align*}
		\KL(p\|q) & = \int_\Omega\int_0^\infty r^{d-1}\left(p(\mx(r, \omega))\log\frac{p(\mx(r, \omega))}{q(\mx(r, \omega))} - p(\mx(r, \omega)) + q(\mx(r, \omega))\right)drd\omega\\
		& = \int_\Omega\int_0^{r_\omega} r^{d-1}\left(p(\mx(r, \omega))\log\frac{p(\mx(r, \omega))}{q(\mx(r, \omega))} - p(\mx(r, \omega)) + q(\mx(r, \omega))\right)drd\omega\\
		&\quad + \int_\Omega\int_{r_\omega}^\infty r^{d-1}\left(p(\mx(r, \omega))\log\frac{p(\mx(r, \omega))}{q(\mx(r, \omega))} - p(\mx(r, \omega)) + q(\mx(r, \omega))\right)drd\omega\\
		& \le \frac{24}{(1-K)^2} H^2(p, q) + \frac{1660032}{(1-K)^3} H^2(p, q) = \frac{1660056}{(1-K)^3} H^2(p, q).
	\end{align*}
	This completes the proof of Theorem \ref{thm-sub-ub}.
\end{proof}

\section{Proof of Theorem \ref{thm-sub-ub-nd}} \label{app-5}
Applying \prettyref{eq:HO} with
$\PP = f_\pi$ and $\SSS = \QQ = f_\eta$, as long as \eqref{eq-condition} holds we get
\begin{equation}\label{eq-ho-sub}
\KL(f_\pi\|f_\eta)\le \frac{4\log(1/\delta)}{(1-\delta)^2}H^2(f_\pi, f_\eta) + \frac{4\delta\log(1/\delta)}{(1-\delta)^2} + \delta^{\frac{\lambda - 1}{2}}\cdot\int_{\RR^d}\frac{f_\pi(\mx)^\lambda}{f_\eta(\mx)^{\lambda - 1}}d\mx.
\end{equation}
Notice that the last term is an $f_\lambda$-divergence $D_{f_\lambda}$ with $f_\lambda(x) = x^\lambda$, which is a convex function for $\lambda\ge 1$, hence $\int_{\RR^d}\frac{f_\pi(\mx)^\lambda}{f_\eta(\mx)^{\lambda - 1}}d\mx$ is convex in $(f_\pi, f_\eta)$. Therefore, by Jensen's inequality we have
\begin{align*}
	\int_{\RR^d}\frac{f_\pi(\mx)^\lambda}{f_\eta(\mx)^{\lambda - 1}}d\mx & = D_{f_\lambda}(\mathbb{E}[\delta_X * \mathcal{N}(0, I_d)]\|\mathbb{E}[\delta_{X'} * \mathcal{N}(0, I_d)])\\
	& \le \mathbb{E}[D_{f_\lambda}(\mathcal{N}(X, I_d)\|\mathcal{N}(X', I_d))]\\
	& = \mathbb{E}\left[\exp\left(\frac{\lambda(\lambda - 1)}{2}\|X - X'\|_2^2\right)\right]
\end{align*}
for any possible coupling between $(X, X')$ where $X\sim \pi, X'\sim \eta$.
\par Now according to the definition of subgaussian distributions, we have
$$\mathbf{P}[\|X\|_2\ge t]\le \exp\left(-\frac{t^2}{2K^2}\right),\quad  \mathbf{P}[\|X'\|_2\ge t]\le \exp\left(-\frac{t^2}{2K^2}\right), \quad \forall t\ge 0,$$
which indicates that $\mathbf{P}[\|X\|_2\ge 2K], \mathbf{P}[\|X'\|_2\ge 2K] < \frac{1}{2}$. Therefore, we can construct the coupling between $(X, X')$ so that if $\|X'\|_2\ge 2K$ we always have $\|X\|_2 < 2K$, and if $\|X\|_2\ge 2K$ we always have $\|X'\|_2 < 2K$. And we have
\begin{align*}
	&\quad\mathbb{E}\left[\exp\left(\frac{\lambda(\lambda - 1)}{2}\|X - X'\|_2^2\right)\right]\\
	& = \mathbb{E}\left[\exp\left(\frac{\lambda(\lambda - 1)}{2}\|X - X'\|_2^2\right)\mathbf{1}_{\|X\|_2\ge 2K}\right] + \mathbb{E}\left[\exp\left(\frac{\lambda(\lambda - 1)}{2}\|X - X'\|_2^2\right)\mathbf{1}_{\|X\|_2 <  2K}\right]\\
	& \le \mathbb{E}\left[\exp\left(\frac{\lambda(\lambda - 1)}{2}\|X - X'\|_2^2\right)\mathbf{1}_{\|X'\|_2 < 2K}\right] + \mathbb{E}\left[\exp\left(\frac{\lambda(\lambda - 1)}{2}\|X - X'\|_2^2\right)\mathbf{1}_{\|X\|_2 <  2K}\right]\\
	& \le \mathbb{E}\left[\exp\left(\frac{\lambda(\lambda - 1)}{2}(\|X\|_2+2K)^2\right)\right] + \mathbb{E}\left[\exp\left(\frac{\lambda(\lambda - 1)}{2}(\|X'\|_2+2K)^2\right)\right]\\
	& \le \mathbb{E}\left[\exp\left(\lambda(\lambda - 1)\left(\|X\|_2^2 + 4K^2\right)\right)\right] + \mathbb{E}\left[\exp\left(\lambda(\lambda - 1)\left(\|X'\|_2^2 + 4K^2\right)\right)\right]
\end{align*}
for any $\delta > 0$.
\par Since $\pi$ and $\eta$ are $K$-subgaussian distributions, we have,
\begin{align*}
	&\quad \mathbb{E}\left[\exp\left(\frac{\|X\|_2^2}{4K^2}\right)\right] = \int_{0}^\infty \exp\left(\frac{t^2}{4K^2}\right)d\pi[\|X\|\le  t]\le \int_{0}^\infty \exp\left(\frac{t^2}{4K^2}\right)\cdot \frac{t}{K^2}\exp\left(-\frac{t^2}{2K^2}\right)dt\\
	& = \int_0^\infty \exp\left(-\frac{t^2}{4K^2}\right)d\left(\frac{t^2}{2K^2}\right) = 2,
\end{align*}
and similarly we also have
$$\mathbb{E}\left[\exp\left(\frac{\|X'\|_2^2}{4K^2}\right)\right]\le 2.$$
We choose $\lambda = \frac{1 + \sqrt{1 + 1/K^2}}{2}$, and we have $\lambda(\lambda - 1) = \frac{1}{4K^2}$. Hence we get
$$\mathbb{E}\left[\exp\left(\frac{\lambda(\lambda - 1)}{2}\|X - X'\|_2^2\right)\right]\le \left(\mathbb{E}\left[\exp\left(\frac{\|X\|_2^2}{4K^2}\right)\right] + \mathbb{E}\left[\exp\left(\frac{\|X'\|_2^2}{4K^2}\right)\right]\right)\cdot \exp(1)\le 4e\le 12.$$
\par Therefore, according to \eqref{eq-ho-sub}, as long as \eqref{eq-condition} holds we get
$$\KL(f_\pi\|f_\eta)\le \frac{4\log(1/\delta)}{(1-\delta)^2}H^2(f_\pi, f_\eta) + \frac{4\delta\log(1/\delta)}{(1-\delta)^2} + 12\delta^{\frac{\lambda - 1}{2}}.$$
Choosing $\delta = \left(\frac{H^2(f_\pi, f_\eta)}{4}\right)^{\frac{8}{(\lambda - 1)^2}\vee 1}$, and we will get
$$\delta\le \frac{H^2(f_\pi, f_\eta)}{4}\le \frac{1}{2}\quad \text{ and }\quad \delta^{\frac{\lambda - 1}{2}}\le H^2(f_\pi, f_\eta).$$
Therefore, the first inequality in \eqref{eq-condition} holds. As for the second one, we notice that if $\lambda\ge 2$, then $\frac{\lambda - 1}{2}\ge \frac{1}{2}$ and it always holds. For $1 < \lambda\le 2$, we have $\frac{8}{(\lambda - 1)^2}\vee 1 = \frac{8}{(\lambda - 1)^2}$, and hence $\log\frac{1}{\delta}\ge \frac{8\log(2)}{(\lambda - 1)^2}\ge \frac{4}{(\lambda - 1)^2} > 4.$ Since $\frac{\log t}{t}$ is decreasing for $t\ge 4$, we have $\frac{\log\log(1/\delta)}{\log(1/\delta)}\le \frac{\log(4/(\lambda - 1)^2)}{4/(\lambda - 1)^2}$. Moreover, using the fact that $x > 2\log x$ holds for all $x\ge 0$, we have $\frac{\lambda - 1}{2}\cdot \frac{4}{(\lambda - 1)^2} - \log \frac{4}{(\lambda - 1)^2}= \frac{2}{\lambda - 1} - 2\log\frac{2}{\lambda - 1} > 0.$
Hence we get $\frac{\log\log(1/\delta)}{\log(1/\delta)}\le \frac{\lambda - 1}{2}$, which proves that the second inequality in \eqref{eq-condition} always holds.

	Therefore, noticing that $(1 - \delta)^2\ge \frac{1}{4}$ and also 
\begin{align*}
	\log(1/\delta)&\le \left(\frac{8}{(\lambda - 1)^2}\vee 1\right)\log\frac{4}{H^2(f_\pi, f_\eta)} = (32K^2(K+\sqrt{K^2+1})^2\vee 1)\log\frac{4}{H^2(f_\pi, f_\eta)}\\
	&\le (512K^4 + 32)\log\frac{4}{H^2(f_\pi, f_\eta)},
\end{align*}
we get
\begin{align*}
	\KL(f_\pi\|f_\eta)\le (10240K^4+652)H^2(f_\pi, f_\eta)\log\frac{4}{H^2(f_\pi, f_\eta)}.
\end{align*}

\section{Comparison inequalities for other distances}\label{sec:other_dist}

In this appendix, we discuss comparison inequalities for other popular distances between densities, namely, the $\chi^2$-divergence, the TV distance, and the $L_2$ distance.

\par First we presents the results of $\chi^2\lesssim H^2$, where $\chi^2(f\|g) = \int \frac{(f - g)^2}{g}$.
\begin{theorem}\label{thm-chi2}
	For $d$-dimensional distributions $\pi, \eta$ supported on $B_2(M)$ with $M\ge 2$, we have
$$\chi^2(f_\pi\|f_\eta)\le 2\exp\left(50(M^2\vee d)\right) H^2(f_\pi, f_\eta).$$
\end{theorem}
\par Next, we show that for one-dimensional Gaussian mixtures where the mixing distribution is compact supported, TV distance and $L_2$ distance are close to each other up to log factors.

\begin{theorem}\label{thm-l2-ub}
	Suppose $\pi$ and $\eta$ are one-dimensional distributions supported on $ [-M, M]$ with $M\ge 1$. Then we have
	$$\TV(f_\pi, f_\eta)\le \left(8\sqrt{M} + 2\log^{1/4}\frac{1}{\|f_\pi- f_\eta\|_2}\right)\|f_\pi - f_\eta\|_2$$
\end{theorem}

\begin{theorem}\label{thm-l2-lb}
	For any one-dimensional distributions $\pi, \eta$ 
 (that need not be compactly supported), 
	$$\|f_\pi - f_\eta\|_2\le \left(\log^{1/4}\frac{1}{\TV(f_\pi, f_\eta)}\vee 3\right)\TV(f_\pi, f_\eta).$$
\end{theorem}

We discuss a statistical application of these results. The $L_2$ squared minimax estimation rates for all Gaussian mixtures are shown in~\cite{kim2014minimax,kim2022minimax} to be $\Theta\left(\log^{d/2} n/n\right)$, which is sharp for all constant $d$. Therefore, equipped with the above comparison theorems, we can also get an upper bound on the minimax estimation rates under the TV distance. 
Previously, 
\cite{ashtiani2020near} showed a rate $\tilde{\mathcal{O}}(\sqrt{kd^2/n})$ for $k$-atomic Gaussian mixtures, where $\tilde{O}$ hides polylog factors. For one-dimensional Gaussian mixtures with compactly supported mixing distributions, the best TV upper bound so far is $\mathcal{O}\left(\log^{3/8} n/\sqrt{n}\right)$, which in fact  follows from combining the sharp $L_2$ rate and Theorem \ref{thm-l2-ub}.

\par More details and proofs are provided in Section \ref{sec-chi2} and \ref{sec-l1-l2}.
Throughout this appendix, for simplicity we abbreviate $p\equiv f_\pi$ and $q\equiv f_\eta$.

\subsection{Proof of Theorem \ref{thm-chi2}}\label{sec-chi2}

\begin{lemma}\label{lem-chi2}
	Suppose $p = \pi * \mathcal{N}(0, I_d), q = \eta * \mathcal{N}(0, I_d)$ where $\supp(\pi), \supp(\eta)\subset B_2(M)$, then for every $r\ge r_\omega\ge M$, we have
	$$\frac{p(\mx(r, \omega))}{q(\mx(r, \omega))}\le \frac{p(\mx(r_\omega, \omega))}{q(\mx(r_\omega, \omega))}\exp\left(2(r - r_\omega)M\right).$$
\end{lemma}
\begin{proof}
	We first prove that for any $\omega\in\Omega$ and $r\ge r_\omega\ge M$,
	$$q(\mx(r, \omega)) \ge q(\mx(r_\omega, \omega))\exp\left(-\frac{(r+M)^2 - (r_\omega+M)^2}{2}\right).$$
	Without loss of generality, we assume $\omega = (1, 0, \cdots, 0)$. Then for any $\mmu = (u_1, u_2, \cdots, u_d)\in B_2(M)$, we have $|u_1|\le M$ and
	\begin{align*}
		\varphi(\mx(r, \omega) - \mmu) & = \frac{1}{\sqrt{2\pi}^d}\exp\left(-\frac{(r - u_1)^2 + \sum_{i=2}^d u_i^2}{2}\right)\\
		\varphi(\mx(r_\omega, \omega) - \mmu) & = \frac{1}{\sqrt{2\pi}^d}\exp\left(-\frac{(r_\omega - u_1)^2 + \sum_{i=2}^d u_i^2}{2}\right).
	\end{align*}
	Noticing that $|u_1|\le M\le r_\omega\le r$, we have $(r-u_1)^2 - (r_\omega-u_1)^2\le (r+M)^2 - (r_\omega+M)^2$, which indicates that
	$$-\frac{(r - u_1)^2 + \sum_{i=2}^d u_i^2}{2}\ge -\frac{(r_\omega - u_1)^2 + \sum_{i=2}^d u_i^2}{2} - \frac{(r+M)^2 - (r_\omega+M)^2}{2},$$
	and hence
	$$\varphi(\mx(r, \omega) - \mmu)\ge \varphi(\mx(r_\omega, \omega) - \mmu)\exp\left(-\frac{(r+M)^2 - (r_\omega+M)^2}{2}\right).$$
	Since we can write
	$$q(\mx(r, \omega)) = \int_{B_2(M)}\eta(\mmu)\varphi(\mx(r, \omega) - \mmu)d\mmu,$$
	we can verify that
	$$q(\mx(r, \omega)) \ge q(\mx(r_\omega, \omega))\exp\left(-\frac{(r+M)^2 - (r_\omega+M)^2}{2}\right).$$
	\par Next, we notice that according to Lemma \ref{lem-eq}, we have
	$$p(\mx(r, \omega)) \le p(\mx(r_\omega, \omega))\exp\left(-\frac{(r-M)^2 - (r_\omega-M)^2}{2}\right).$$
	This indicates that
	\begin{align*}
		\frac{p(\mx(r, \omega))}{q(\mx(r, \omega))} & \le \frac{p(\mx(r_\omega, \omega))}{q(\mx(r_\omega, \omega))}\exp\left(-\frac{(r+M)^2 - (r_\omega+M)^2}{2} + \frac{(r-M)^2 - (r_\omega-M)^2}{2}\right)\\
		& = \frac{p(\mx(r_\omega, \omega))}{q(\mx(r_\omega, \omega))}\exp\left(2(r - r_\omega)M\right).
	\end{align*}
\end{proof}

\begin{proof}[Proof of Theorem \ref{thm-chi2}] Without loss of generality, we assume $d\le M^2$. 
We write
	\begin{equation}\label{eq-chi2-polar}
		\begin{aligned}
		\chi^2(p\|q) & = \int_\Omega\int_0^\infty r^{d-1}q(\mx(r, \omega)\left(\frac{p(\mx(r, \omega))}{q(\mx(r, \omega))} - 1\right)^2drd\omega\\
		H^2(p, q) & = \frac{1}{2}\int_\Omega\int_0^\infty r^{d-1}q(\mx(r, \omega))\left(\sqrt{\frac{p(\mx(r, \omega))}{q(\mx(r, \omega))}} - 1\right)^2drd\omega
	\end{aligned}\end{equation}
	For every $\omega\in\Omega$, we define $r_\omega$ as
	$$r_\omega\triangleq \inf\left\{r\bigg|\frac{p(\mx(r_\omega, \omega))}{q(\mx(r_\omega, \omega))}\ge \exp\left(25M^2\right)\right\}.$$
	Notice that for any $r\le 6M$ and $\omega\in\Omega$, we have
	\begin{align*}
		p(\mx(r, \omega)) & = \int_{B_2(M)}\pi(\mmu)\varphi(\mx(r, \omega), \mmu)d\mmu\le \frac{1}{\sqrt{2\pi}^d}\\
		q(\mx(r, \omega)) & = \int_{B_2(M)}\eta(\mmu)\varphi(\mx(r, \omega), \mmu)d\mmu\ge \frac{1}{\sqrt{2\pi}^d}\exp\left(-\frac{(6M + M)^2}{2}\right),
	\end{align*}
	which indicates that
	$$\frac{p(\mx(r, \omega))}{q(\mx(r, \omega)}\le \exp\left(\frac{49M^2}{2}\right) < \exp\left(25M^2\right).$$
	Hence for every $\omega\in \Omega$, we all have $r_\omega\ge 6M$. And if $r_\omega\neq \infty$, we have that $\frac{p(\mx(r_\omega, \omega))}{q(\mx(r_\omega, \omega))} = \exp\left(25M^2\right)$. According to Lemma \ref{lem-eq} and Lemma \ref{lem-chi2}, we know that for every $r\ge r_\omega$, we all have
	\begin{align*}
		q(\mx(r, \omega)) & \le q(\mx(r_\omega, \omega))\exp\left(-\frac{(r-M)^2 - (r_\omega-M)^2}{2}\right),\\
		\frac{p(\mx(r, \omega))}{q(\mx(r, \omega))} & \le \frac{p(\mx(r_\omega, \omega))}{q(\mx(r_\omega, \omega))}\exp\left(2(r - r_\omega)M\right).\\
	\end{align*}
	Since $\frac{p(\mx(r_\omega, \omega))}{q(\mx(r_\omega, \omega))} = \exp(25M^2)\ge 1$, and $\exp\left(2(r - r_\omega)M\right)\ge 1$ for every $r\ge r_\omega$, we have
	$$\left(\frac{p(\mx(r, \omega))}{q(\mx(r, \omega))} - 1\right)^2\le \max\left\{1, \left(\frac{p(\mx(r_\omega, \omega))}{q(\mx(r_\omega, \omega))}\exp\left(2(r - r_\omega)M\right)\right)^2\right\} = \exp\left(50M^2 + 4(r - r_\omega)M\right).$$
	Therefore, we obtain that
	$$\begin{aligned}
		&\quad \int_{r_\omega}^\infty r^{d-1}q(\mx(r, \omega))\cdot\left(\frac{p(\mx(r, \omega))}{q(\mx(r, \omega))} - 1\right)^2dr\\
		&\le q(\mx(r_\omega, \omega))\int_{r_\omega}^\infty r^{d-1}\exp\left(50M^2 + 4(r - r_\omega)M-\frac{(r-M)^2 - (r_\omega-M)^2}{2}\right)dr.
	\end{aligned}$$
	We adopt the changes of variables from $r$ to $t = r - r_\omega$, and obtain that
	$$\begin{aligned}
		&\quad \int_{r_\omega}^\infty r^{d-1}\exp\left(50M^2 + 4(r - r_\omega)M-\frac{(r-M)^2 - (r_\omega-M)^2}{2}\right)dr\\
		& = r_\omega^{d-1}\exp(50M^2)\int_0^\infty\exp\left((d-1)\log\frac{t + r_\omega}{r_\omega} - \frac{t^2}{2} - (r_\omega - M)t + 4Mt\right)dt.
	\end{aligned}$$
	\par Next, we define
	$$f(t) \triangleq (d-1)\log\frac{t + r_\omega}{r_\omega} - \frac{1}{2}(r_\omega - 5M)t.$$
	Since 
	$$d - 1 < d\le M^2 < \frac{1}{2}\cdot (6M)\cdot (6M - 5M)\le \frac{1}{2}r_\omega(r_\omega - 5M),$$
	the first-order derivative of $f$ satisfies that 
	$$f'(t) = \frac{d-1}{t + r_\omega} - \frac{1}{2}(r_\omega - 5M)\le \frac{d-1}{r_\omega} - \frac{1}{2}(r_\omega - 5M)\le 0, \quad \forall t\ge 0.$$
	Therefore we have
	$$(d-1)\log\frac{t + r_\omega}{r_\omega} - \frac{1}{2}(r_\omega - 5M)t = f(t)\le f(0) = 0,\quad \forall t\ge 0.$$
	This directly indicates that
	$$\begin{aligned}
		&\quad \int_0^\infty\exp\left((d-1)\log\frac{t + r_\omega}{r_\omega} - \frac{t^2}{2} - (r_\omega - M)t + 4Mt\right)dt\\
		& \le \int_0^\infty\exp\left(-\frac{t^2}{2} - \frac{1}{2}(r_\omega - 5M)t\right)dt\le \int_0^\infty\exp\left(-\frac{1}{2}(r_\omega - 5M)t\right)dt = \frac{2}{r_\omega - 5M}
	\end{aligned}$$
	Therefore, we obtain that
	$$\int_{r_\omega}^\infty r^{d-1}q(\mx(r, \omega))\cdot\left(\frac{p(\mx(r, \omega))}{q(\mx(r, \omega))} - 1\right)^2dr\le \frac{2r_\omega^{d-1}\exp(50M^2)q(\mx(r_\omega, \omega))}{r_\omega - 5M}$$

	\par Next, according to Lemma \ref{lem-eq2}, for $\forall \omega\in \Omega$ and $r\in [0, r_\omega]$ we have
	$$\nabla \log\frac{p(\mx(r, \omega))}{q(\mx(r, \omega))} = \nabla \log p(\mx(r, \omega)) - \nabla \log q(\mx(r, \omega))\le 6r + 8M\le 6r_\omega + 8M\le  8r_\omega + 8M.$$
	Notice that we also have $\log\frac{p(\mx(r_\omega, \omega))}{q(\mx(r_\omega, \omega))} = 25M^2$. Hence for $\forall r_\omega - \frac{1}{r_\omega + M}\le r\le r_\omega$, 
	$$\log\frac{p(\mx(r, \omega))}{q(\mx(r, \omega))}\ge 25M^2 - (8r_\omega + 8M)\cdot \frac{1}{r_\omega + M} = 25M^2 - 8\ge 2,$$
	which indicates that
	$$\left(\sqrt{\frac{p(\mx(r, \omega))}{q(\mx(r, \omega))}} - 1\right)^2\ge \left(e^2 - 1\right)^2\ge 40.$$
	We further notice that $r_\omega\ge 5M\ge M$ and $d-1 < d\le M^2$. Hence for $\forall r_\omega - \frac{1}{r_\omega + M}\le r\le r_\omega$, 
	$$r^{d-1}\ge r_\omega^{d-1}\cdot \left(1 - \frac{1}{r_\omega(r_\omega + M)}\right)^{d-1}\ge r_\omega^{d-1}\cdot \left(1 - \frac{d-1}{r_\omega(r_\omega + M)}\right)\ge r_\omega^{d-1}\left(1 - \frac{d-1}{30M^2}\right)\ge \frac{1}{2}r_\omega^{d-1}.$$
	After noticing that $p(\mx(r, \omega))\ge p(\mx(r_\omega, \omega))$ according to Lemma \ref{lem-eq}, we obtain
	\begin{align*}
		&\quad \int_{r_\omega - \frac{1}{r_\omega + M}}^{r_\omega}r^{d-1}q(\mx(r, \omega))\cdot \left(\sqrt{\frac{p(\mx(r, \omega))}{q(\mx(r, \omega))}} - 1\right)^2dr\\
		& \ge \frac{1}{r_\omega + M}\cdot \min_{r_\omega - \frac{1}{r_\omega + M}\le r\le r_\omega}r^{d-1}q(\mx(r, \omega))\cdot \left(\sqrt{\frac{p(\mx(r, \omega))}{q(\mx(r, \omega))}} - 1\right)^2\\
		&\ge \frac{1}{r_\omega + M}\cdot \frac{1}{2}r_\omega^{d-1}\cdot 40q(\mx(r_\omega, \omega))\ge \frac{2r_\omega^{d-1}q(\mx(r_\omega, \omega))}{r_\omega - 5M}\\
		& \ge \exp\left(-50M^2\right)\int_{r_\omega}^\infty r^{d-1}q(\mx(r, \omega))\cdot\left(\frac{p(\mx(r, \omega))}{q(\mx(r, \omega))} - 1\right)^2dr.
	\end{align*}
	Therefore, according to \eqref{eq-chi2-polar}, we have
	\begin{align*}
		H^2(p, q) & = \int_\Omega\int_0^\infty r^{d-1}p(\mx(r, \omega))\left(\sqrt{\frac{q(\mx(r, \omega))}{p(\mx(r, \omega))}} - 1\right)^2drd\omega\\
		& \ge \int_\Omega\int_{r_\omega - \frac{1}{r_\omega + M}}^{r_\omega} r^{d-1}p(\mx(r, \omega))\left(\sqrt{\frac{q(\mx(r, \omega))}{p(\mx(r, \omega))}} - 1\right)^2drd\omega\\
		& \ge \exp\left(-50M^2\right)\int_\Omega\int_{r_\omega}^\infty r^{d-1}q(\mx(r, \omega))\cdot\left(\frac{p(\mx(r, \omega))}{q(\mx(r, \omega))} - 1\right)^2drd\omega.
	\end{align*}

	\par Next we consider those $\mx(r, \omega)$ with $0\le r\le r_\omega$. According to the definition of $r_\omega$, we know that for any such $r$, we have 
	$$\frac{p(\mx(r, \omega))}{q(\mx(r, \omega))}\le \exp\left(25M^2\right).$$
	Notice the inequality
	$$(t-1)^2 = (\sqrt{t} + 1)^2(\sqrt{t} - 1)^2\le \exp\left(50M^2\right)\left(\sqrt{t} - 1\right)^2, \quad \forall 0\le t\le \exp\left(25M^2\right).$$
	Hence we obtain that
	\begin{align*}
		&\quad\int_{0}^{r_\omega} r^{d-1}q(\mx(r, \omega))\cdot\left(\frac{p(\mx(r, \omega))}{q(\mx(r, \omega))} - 1\right)^2dr\\
		& \le \int_0^{r_\omega} r^{d-1}q(\mx(r, \omega))\cdot \exp\left(50M^2\right)\left(\sqrt{\frac{p(\mx(r, \omega))}{q(\mx(r, \omega))}} - 1\right)^2dr.
	\end{align*}
	Therefore, we have
	\begin{align*}
		&\quad \int_\Omega\int_{0}^{r_\omega} r^{d-1}q(\mx(r, \omega))\cdot\left(\frac{p(\mx(r, \omega))}{q(\mx(r, \omega))} - 1\right)^2drd\omega\\
		&\le \exp\left(50M^2\right)\int_\Omega\int_0^{r_\omega} r^{d-1}q(\mx(r, \omega))\cdot\left(\sqrt{\frac{p(\mx(r, \omega))}{q(\mx(r, \omega))}} - 1\right)^2drd\omega\\
		& \le \exp\left(50M^2\right)\int_\Omega\int_0^{\infty} r^{d-1}q(\mx(r, \omega))\cdot\left(\sqrt{\frac{p(\mx(r, \omega))}{q(\mx(r, \omega))}} - 1\right)^2drd\omega = \exp\left(50M^2\right)H^2(p, q).
	\end{align*}
	\par Combine these two analysis together, we obtain that
	\begin{align*}
		\KL(p\|q) & = \int_\Omega\int_0^{\infty} r^{d-1}q(\mx(r, \omega))\cdot\left(\sqrt{\frac{p(\mx(r, \omega))}{q(\mx(r, \omega))}} - 1\right)^2drd\omega\\
		& = \int_\Omega\int_0^{r_\omega} r^{d-1}q(\mx(r, \omega))\cdot\left(\sqrt{\frac{p(\mx(r, \omega))}{q(\mx(r, \omega))}} - 1\right)^2drd\omega\\
		&\quad + \int_\Omega\int_{r_\omega}^{\infty} r^{d-1}q(\mx(r, \omega))\cdot\left(\sqrt{\frac{p(\mx(r, \omega))}{q(\mx(r, \omega))}} - 1\right)^2drd\omega\\
		& \le \exp\left(50M^2\right) H^2(p, q) + \exp\left(50M^2\right) H^2(p, q) = 2\exp\left(50M^2\right)H^2(p, q).
	\end{align*}
	This completes the proof of Theorem \ref{thm-chi2}.
\end{proof}

\subsection{Proof of Theorem \ref{thm-l2-ub} and \ref{thm-l2-lb}}\label{sec-l1-l2}
\begin{proof}[Proof of Theorem \ref{thm-l2-ub}]
	First we notice that for any $|t|\ge M$, we have
	$$0\le p(t) = \int_{|x|\ge M} \varphi(t - x)d\pi(x)\le \max_{|x|\ge M}\varphi(t - x) = \frac{1}{\sqrt{2\pi}}\exp\left(-\frac{(|t| - M)^2}{2}\right).$$
	Similarly we have the same estimation for $q(t)$. Hence we get
	$$|p(t) - q(t)|\le \frac{1}{\sqrt{2\pi}}\exp\left(-\frac{(|t| - M)^2}{2}\right) \quad \forall |t|\ge M,$$
	which indicates that for any $m\ge M$
	\begin{align*}
		\int_{|x|\ge t}|p(x) - q(x)|dx & \le \frac{1}{\sqrt{2\pi}}\int_{|x|_2\ge t} \exp\left(-\frac{(|x| - M)^2}{2}\right)ds\\
		& = \frac{2}{\sqrt{2\pi}}\int_{t}^\infty \exp\left(-\frac{(s-M)^2}{2}\right)ds\le \frac{2}{t-M}\exp\left(-\frac{(t-M)^2}{2}\right).
	\end{align*}
	Hence for $t\ge M+\frac{1}{3}$ we have
	$$\int_{|x|\ge t}|p(x) - q(x)|dx\le 6\exp\left(-\frac{(t-M)^2}{2}\right).$$
	\par Additionally, according to Cauchy-Schwarz inequality we have
	$$\left(\int_{|x|\le t}|p(t) - q(t)|dt\right)^2\le 2t\cdot \left(\int_{|x|\le t}|p(x) - q(x)|^2dx\right),$$
	which indicates that
	$$\int_{|x|\le t}|p(x) - q(x)|dx\le \sqrt{2t}\cdot \sqrt{\int_{|x|\le t}|p(x) - q(x)|^2dx}\le \sqrt{2t}\cdot \|p - q\|_2.$$
	Therefore, we obtain that for every $t\ge M + \frac{1}{3}$,
	$$\TV(p, q) = \int_{|x|\ge t}|p(x) - q(x)|dx + \int_{|x|\le t}|p(x) - q(x)|dx\le 6\exp\left(-\frac{(t-M)^2}{2}\right) + \sqrt{2t}\cdot \|p - q\|_2.$$
	\par Finally, since for any $x\in\RR$, we have $0\le p(x), q(x)\le \max_{x\in\RR^d}\varphi(x) = \frac{1}{\sqrt{2\pi}}\le \frac{2}{5}$, we obtain that
	$$\|p - q\|_2^2 = \int_{-\infty}^\infty (p(x) - q(x))^2dx\le \int_{-\infty}^\infty \frac{2(p(x) + q(x))}{5}dx = \frac{4}{5},$$
	which indicates that
	$$\log\frac{1}{\|p - q\|_2} = \frac{1}{2}\log\frac{1}{\|p - q\|_2^2}\ge \frac{1}{2}\log\frac{5}{4}\ge \frac{1}{10}.$$
	Therefore, choosing 
	$$t = M + \sqrt{2\log\frac{1}{\|p - q\|_2}}\ge M + \sqrt{\frac{1}{5}}\ge M + \frac{1}{3},$$
	we get
	\begin{align*}
		\TV(p, q) & \le 6\|p - q\|_2 + \sqrt{2M + 2\sqrt{2\log\frac{1}{\|p - q\|_2}}}\|p - q\|_2\\
		& \le 6\|p - q\|_2 + \left(\sqrt{2M} + \sqrt{2\sqrt{2\log\frac{1}{\|p - q\|_2}}}\right)\|p - q\|_2\\
		& \le \left(8\sqrt{M} + 2\log^{1/4}\frac{1}{\|p - q\|_2}\right)\|p - q\|_2.
	\end{align*}
\end{proof}

\begin{proof}[Proof of Theorem \ref{thm-l2-lb}]
  For any distribution $\PP$, we define its characteristic function $\Psi_\PP: \RR\to \mathbb{C}$ as:
	$$\Psi_\PP(t)\triangleq \mathbb{E}\left[e^{itX}\right], \quad X\sim\PP.$$
 
	\par Suppose the characteristic function of $\pi, \eta, p, q$ are $\Psi_\pi, \Psi_\eta, \Psi_p, \Psi_q$, respectively. 
 Then by Gaussian convolution 
	$$\Psi_p(t) = \Psi_\pi(t) \exp\left(-\frac{t^2}{2}\right), \quad \Psi_q(t) = \Psi_\eta(t) \exp\left(-\frac{t^2}{2}\right).$$
	We notice that for every $t\in\RR$, we have from 
 Plancherel's identity
	$$|\Psi_p(t) - \Psi_q(t)| = \left|\int_{-\infty}^\infty (p(x) - q(x))e^{itx}dx\right|\le \int_{-\infty}^\infty |(p(x) - q(x))e^{itx}|dx = \TV(p, q).$$
	Similarly, we also have $|\Psi_\pi(t) - \Psi_\eta(t)|\le \|\pi - \eta\|_1\le 2$, which indicates that for every $t\in\RR$, 
	$$|\Psi_p(t) - \Psi_q(t)| = e^{-\frac{t^2}{2}}|\Psi_\pi(t) - \Psi_\eta(t)|\le 2\exp\left(-\frac{t^2}{2}\right).$$
	Therefore, we obtain that for any $s > 0$, 
	\begin{align*}
		\|\Psi_p - \Psi_q\|_2^2 & = \int_{-\infty}^\infty \left|\Psi_p(t) - \Psi_q(t)\right|^2dt\\
		& = \int_{-\infty}^{-s} \left|\Psi_p(t) - \Psi_q(t)\right|^2dt + \int_{-s}^s \left|\Psi_p(t) - \Psi_q(t)\right|^2dt + \int_{s}^\infty \left|\Psi_p(t) - \Psi_q(t)\right|^2dt\\
		& \le \int_{-\infty}^{-s}\exp\left(-\frac{t^2}{2}\right)dt + \int_{s}^{\infty}\exp\left(-\frac{t^2}{2}\right)dt + 2s\cdot \TV(p, q)^2\\
		& = 2\int_0^\infty\exp\left(-\frac{(t+s)^2}{2}\right)dt + 2s\cdot \TV(p, q)^2 = \frac{2}{s}\exp\left(-\frac{s^2}{2}\right) + 2s\cdot \TV(p, q)^2
	\end{align*}
	\par When $\TV(p, q)\ge \frac{1}{e}$. We further notice that $|\Psi_p(t) - \Psi_q(t)|\le 2\exp\left(-\frac{t^2}{2}\right)$, which indicates that
	$$\|\Psi_p - \Psi_q\|_2^2 = \int_{-\infty}^\infty |\Psi_p(t) - \Psi_q(t)|^2dt\le \int_{-\infty}^\infty 4\exp(-t^2)dt = 4\sqrt{\pi}\le \frac{36}{e}\le 36\TV(p, q)$$
	When $\TV(p, q)\le \frac{1}{e}$, we have $s = \sqrt{2\log\frac{1}{\TV(p, q)^2}} = 2\sqrt{\log\frac{1}{\TV(p, q)}}\ge 2$, we get
	$$\|\Psi_p - \Psi_q\|_2^2\le \left(1 + 2\sqrt{\log\frac{1}{\TV(p, q)}}\right)\cdot \TV(p, q)^2\le 4\sqrt{\log\frac{1}{\TV(p, q)}}\cdot \TV(p, q)^2.$$
	Above all, we get
	$$\|\Psi_p - \Psi_q\|_2^2\le 4\left(\log^{1/4}\frac{1}{\TV(p, q)}\vee 3\right)^2\cdot \TV(p, q),$$
	which indicates that
	$$\|p - q\|_2^2 = \sqrt{\frac{1}{2\pi}\|\Psi_p - \Psi_q\|_2}\le \left(\log^{1/4}\frac{1}{\TV(p, q)}\vee 3\right)\TV(p, q).$$
\end{proof}

\end{document}